\documentclass[table]{amsart}
\usepackage[margin=1.0in]{geometry}
\usepackage{amscd,amsmath,amsxtra,amsthm,amssymb,stmaryrd,xr,mathrsfs,mathtools,enumerate,commath, comment, enumitem, graphicx}

\usepackage{stmaryrd}
\usepackage{orcidlink}
\usepackage{multirow}
\usepackage{xcolor}
\usepackage{commath}
\usepackage{comment}
\usepackage{tikz-cd}
\usepackage{tkz-graph}

\usepackage{longtable} 
\usepackage{pdflscape} 
\usepackage{booktabs}
\usepackage{hyperref}
\definecolor{vegasgold}{rgb}{0.77, 0.7, 0.35}
\definecolor{darkgoldenrod}{rgb}{0.72, 0.53, 0.04}
\definecolor{gold(metallic)}{rgb}{0.83, 0.69, 0.22}
\hypersetup{
 colorlinks=true,
 linkcolor=darkgoldenrod,
 filecolor=brown,      
 urlcolor=gold(metallic),
 citecolor=darkgoldenrod,
 }
\newtheorem{lthm}{Theorem}

\usepackage[all,cmtip]{xy}

\newcommand\tinycube{\begin{tikzpicture}[scale=1.5]
    \coordinate (A1) at (0, 0);
    \coordinate (A2) at (0, 0.1);
    \coordinate (A3) at (0.1, 0.1);
    \coordinate (A4) at (0.1, 0);
    \coordinate (B1) at (0.03, 0.03);
    \coordinate (B2) at (0.03, 0.13);
    \coordinate (B3) at (0.13, 0.13);
    \coordinate (B4) at (0.13, 0.03);

    \draw (A1) -- (A2);
    \draw (A2) -- (A3);
    \draw (A3) -- (A4);
    \draw (A4) -- (A1);
    \draw[densely dotted] (A1) -- (B1);
    \draw[densely dotted] (B1) -- (B2);
    \draw (A2) -- (B2);
    \draw (B2) -- (B3);
    \draw (A3) -- (B3);
    \draw (A4) -- (B4);
    \draw (B4) -- (B3);
    \draw[densely dotted] (B1) -- (B4);
\end{tikzpicture}}

\DeclareFontFamily{U}{wncy}{}
\DeclareFontShape{U}{wncy}{m}{n}{<->wncyr10}{}
\DeclareSymbolFont{mcy}{U}{wncy}{m}{n}
\DeclareMathSymbol{\Sh}{\mathord}{mcy}{"58}
\usepackage[T2A,T1]{fontenc}
\usepackage[OT2,T1]{fontenc}

\usepackage{tikz}
\usetikzlibrary{shapes.geometric}
\usetikzlibrary{decorations.markings}
\tikzset{every loop/.style={min distance=10mm,looseness=10}}
\tikzstyle{vertex}=[auto=left,circle,minimum size=1pt,inner sep=0pt]

\newtheorem{theorem}{Theorem}[section]
\newtheorem{lemma}[theorem]{Lemma}

\newtheorem*{theorem*}{Theorem}
\newtheorem*{ass*}{Assumption}

\newtheorem{remark}[theorem]{Remark}

\newtheorem{proposition}[theorem]{Proposition}

\newcommand{\fq}{\mathfrak{q}}

\newcommand{\Q}{\mathbb{Q}}
\newcommand{\F}{\mathbb{F}}

\newcommand{\cL}{\mathcal{L}}

\newcommand{\cM}{\mathcal{M}}

\newcommand{\op}[1]{\operatorname{#1}}

\newcommand\numberthis{\addtocounter{equation}{1}\tag{\theequation}}

\numberwithin{equation}{section}

\begin{document}

\title[Second moments for cubic L-functions]{The second moment of cubic Dirichlet L-functions over function fields}

\author[S.~Goel]{Shivani Goel\, \orcidlink{0009-0000-7841-2011}}
\address[Goel]{Chennai Mathematical Institute, H1, SIPCOT IT Park, Kelambakkam, Siruseri, Tamil Nadu 603103, India}
\email{shivanig@cmi.ac.in}

\author[A.~Ray]{Anwesh Ray\, \orcidlink{0000-0001-6946-1559}}
\address[Ray]{Chennai Mathematical Institute, H1, SIPCOT IT Park, Kelambakkam, Siruseri, Tamil Nadu 603103, India}
\email{anwesh@cmi.ac.in}

\keywords{cubic L-functions over function fields, moments, Gauss sums}
\subjclass[2020]{11R59, 11M38, 11R58}

\begin{abstract}
In this article, we study the second moment of cubic Dirichlet L-functions at the central point $s=1/2$ over the rational function field $\mathbb{F}_q(T)$, where $q$ is a power of an odd prime satisfying $q \equiv 2 \pmod{3}$. Our result extends prior work of David, Florea and Lalin \cite{DFLCubic}, who obtained an asymptotic formula for the first moment. Our approach relies on analytic techniques (Perron's formula, approximate functional equation, etc), adapted to the function field context.  A key step in the construction is to relate second moment to certain averages of Gauss sums, which are estimated in \emph{loc. cit.} using results of Kubota \cite{kubota} and Hoffstein \cite{Hoffsteinmetaplectic}. 
\end{abstract}

\maketitle

\section{Introduction}
\subsection{Background and motivation}
\par The study of L-functions in the function field setting has long served as a fertile testing ground for deep conjectures in number theory and arithmetic geometry. The arithmetic of the polynomial ring \( \mathbb{F}_q[T] \), where \( q \) is a power of a fixed odd prime \( p \), closely parallels the classical theory of the integers, with zeta and \( L \)-functions playing an analogous role. In this setting, one gains access to powerful tools from algebraic geometry, most notably through the interpretation of \( L \)-functions as generating series encoding the action of the Frobenius endomorphism on the \'etale cohomology of algebraic curves over finite fields. For a systematic introduction to the topic, the reader may consult \cite{Rosenfunctionfields, weilconj, ThakurFFA}.

The study of moments of Dirichlet \( L \)-functions at the central point \( \tfrac{1}{2} \) is a fundamental problem in analytic number theory. For quadratic characters \( \chi \) over the rationals, the first and second moments of \( L(\tfrac{1}{2}, \chi) \) were computed by Jutila \cite{Jutila}, who established the asymptotic formula
\[
\sum_{0<d \leqslant D} L\left(\tfrac{1}{2}, \chi_{d}\right)= \frac{P(1)}{4 \zeta(2)} D\left\{\log \left(\frac{D}{\pi}\right)+\frac{\Gamma^{\prime}}{\Gamma}\left(\tfrac{1}{4}\right)+4 \gamma - 1 + 4 \frac{P^{\prime}}{P}(1)\right\} +O\left(D^{3/4+\varepsilon}\right),
\]
where $d$ ranges over squarefree numbers between $1$ and $D$, and $P(s):=\prod_{p}\left(1-\frac{1}{(p+1)p^{s}}\right)$.
Moreover, for the second moment, he showed that
\[
\sum_{0<d \leqslant D} L\left(\tfrac{1}{2}, \chi_{d}\right)^{2}=\frac{c_1}{\zeta(2)} D \log^{3} D + O\left(D(\log D)^{5/2+\varepsilon}\right),
\]
for an explicit constant $c_1>0$.
Related results for the third moment were obtained by Soundararajan \cite{Sound}. There are conjectural asymptotics for moments of quadratic Dirichlet $L$-functions over $\Q$ due to Keating and Snaith \cite{keatingsnaith} based on random matrix theory. Baier and Young \cite{baieryoung} obtained the smoothed first moment for cubic Dirichlet characters over $\Q$. Certain results have also been obtained over $\Q$ for thin families of characters by Luo \cite{Luo}, Friedberg--Hoffstein--Lieman \cite{FHL}. More recently, David, de Faveri, Dunn, and Stucky \cite{david2024non} have computed the first and second moments of $L(1/2, \chi_q)$ for primitive cubic characters $\chi_q$ with $q \in \mathbb{Z}[\mu_3]$ satisfying $q \equiv 1 \pmod{9}$, thereby establishing an unconditional positive proportion of non-vanishing at the central point.
\par Over the rational function field, the first results of this nature are due to Hoffstein and Rosen \cite{HoffsteinRosen}, who obtained an asymptotic formula for the first moment of $L(\frac{1}{2}, \chi)$ where $\chi$ ranges over quadratic Dirichlet characters of the rational function field $\F_q(T)$. Such $L$-functions are closely related to the zeta functions of hyperelliptic curves over $\F_q$. Their methods involve the use of Eisenstein series for the metaplectic $2$-fold cover of $\op{GL}_2$. A different proof of this result, using classical methods was obtained by Andrade and Keating \cite{andradekeating}. This classical approach was further extended by Florea, who obtained asymptotic formulae for the second, third and fourth moments \cite{Florea2ndand3rd, Florea4th}. 
\subsection{Main results}
\par For cubic characters over the rational function field $\F_q(T)$, an asymptotic formula for the first moment was obtained by David–Florea–Lalin \cite{DFLCubic}. Let $L_q(s, \chi)$ denote the $L$-function associated to a character $\chi$ of $\F_q(T)$. A natural dichotomy arises depending on whether $q \equiv 1 \pmod{3}$ or $q \equiv 2 \pmod{3}$, since in the former case, the third roots of unity are contained in $\F_q^\times$. Let $\mu_3$ (resp. $\mu_{3,q}$) denote the group of third roots of unity in $\mathbb{C}^\times$ (resp. in $\F_q^\times$), and fix a bijection $\Omega: \mu_3 \xrightarrow{\sim} \mu_{3,q}$. Define a character $\chi_3$ on $\F_q^\times$ by setting $\chi_3(a) := \Omega^{-1}(a^{\frac{q-1}{3}})$. One refers to the case $q \equiv 1 \pmod{3}$ as the \emph{Kummer case}, and to $q \equiv 2 \pmod{3}$ as the \emph{non-Kummer case}.

\par When $q \equiv 2 \pmod{3}$, Theorem 1.1 of \emph{loc. cit.} asserts that:
\[
\sum_{\substack{\chi \text{ primitive cubic} \\ \mathrm{genus}(\chi) = g}} L_q(1/2, \chi)
= \frac{\zeta_q(3/2)}{\zeta_q(3)} C(q) q^{g+2} + O(q^{\frac{7g}{8} + \varepsilon g}),
\]
where $C(q)$ is an explicit constant depending only on $q$ and $\varepsilon>0$. On the other hand when $q \equiv 1 \pmod{3}$, their result applies to primitive cubic characters $\chi$ satisfying the condition $\chi|_{\F_q^\times} = \chi_3$. Theorem 1.2 of \emph{loc. cit.} states that there exist explicit constants $C_1(q), C_2(q) > 0$ such that for $\varepsilon>0$,
\[
\sum_{\substack{\chi \text{ primitive cubic} \\ \mathrm{genus}(\chi) = g \\ \chi|_{\F_q^\times} = \chi_3}} L_q(1/2, \chi)
= C_1(q) g q^{g} + C_2(q) q^{g} + O\left(q^{g\frac{1 + \sqrt{7}}{4} + \varepsilon g}\right).
\]
The hypothesis that $\chi$ restricts to $\chi_3$ on $\F_q^\times$ serves to simplify the analysis by ensuring that all $L$-functions in the family satisfy a uniform functional equation which is easier to work with. Restrictions of this nature are common in the field, and help simplify many calculations. In \cite{andradekeating, Florea2ndand3rd, Florea4th}, the assumption that $q\equiv 1\pmod{4}$ is made. Furthermore, in the latter two works, $q$ is assumed to be a prime.

\par In this article we consider the non-Kummer case, i.e., we assume that $q$ is an odd prime power and $q\equiv 2\pmod{3}$. We obtain an asymptotic for the second moment of $L_q(\frac{1}{2}, \chi)$. 

\begin{lthm}\label{main theorem}
   With respect to notation above, assume that $q \equiv 2 \pmod{3}$, then, 
\begin{equation}\label{eqn for main theorem}
\sum_{\substack{\chi \text{ primitive cubic} \\ \mathrm{genus}(\chi) = g}} L_q(1/2, \chi)^2
= \frac{g(g+2)A_q(\frac{1}{q^2}, \frac{1}{q^{3/2}}) \zeta_{q}(3/2)^2}{8\zeta_{q}(3)}q^{g+2} + O(q^{g}),
\end{equation}
where $A_q$ is the function defined in \eqref{A_Qdefn}.
\end{lthm}
\noindent The case $q \equiv 1 \pmod{3}$ is not addressed in this paper, as it necessitates the analysis of certain twisted averages of Gauss sums, which have not been studied in the literature. The difficulties arising in this case are substantial and will be addressed in a more general framework in forthcoming work.

\subsection{Methodology}
Let us briefly outline the strategy of the proof of Theorem \ref{main theorem}. For any positive integer \( n \), let \( \mathcal{M}_{q,n} \subset \F_q[T] \) be the set of all monic polynomials of degree \( n \). Assume henceforth that $q\equiv 2\pmod{3}$. Let \( \mathcal{M}_q \subset \F_q[T] \) denote the set of all monic polynomials, and let \( \mathcal{M}_{q,\leq n} \subset \mathcal{M}_q \) denote the subset of monic polynomials of degree \( \leq n \). Define \( \mathcal{H}_{q,n} \) as the subset of \( \mathcal{M}_{q,n} \) consisting of squarefree monic polynomials. It is easy to see that
    \[
    \# \mathcal{M}_{q,n} = q^n, \quad \# \mathcal{M}_{q,\leq n} = 1 + q + \dots + q^n = \frac{q^{n+1} - 1}{q - 1}, \quad \# \mathcal{H}_{q,n} = q^n \left(1 - \frac{1}{q} \right).
    \] Let $\chi$ be a cubic character and $L_q(s, \chi)$ be the associated $L$-function over $\F_q[T]$. Much like number field case, \( L_q(s, \chi) \) satisfies a functional equation (see Lemma \ref{functional equation lemma}), which yields the \emph{approximate functional equation}:
\begin{equation}\label{approx fe intro}
\begin{split}
L_q(1/2, \chi)^2 = \left(1 - \frac{1}{\sqrt{q}}\right)^2 \Bigg( & \sum_{i=0}^A (i+1) \sum_{f \in \cM_{\leq A - i}} \frac{\chi(f)\, d(f)}{\sqrt{|f|}} \\
& + \sum_{i=0}^{2g - A - 1} (i+1)\, \omega(\chi)^2 \sum_{f \in \cM_{\leq 2g - A - 1 - i}} \frac{\bar{\chi}(f)\, d(f)}{\sqrt{|f|}} \Bigg).
\end{split}
\end{equation}
Here, $d(f)$ is the divisor function, $\omega(\chi)$ is the sign of $\chi$ and the choice of $A$ is postponed till the end of the argument. We refer to Proposition \ref{approx funct eqn} for further details. Associated to a polynomial $F \in \mathcal{H}_{q^2, g/2 + 1}$, we have a primitive cubic Dirichlet character $\chi_F$ defined in section \ref{section 2.5}.
The second moment of central L-values over primitive cubic characters of genus \( g \) is then given by
\[
\sum_{\substack{\chi \text{ primitive cubic} \\ \mathrm{genus}(\chi) = g}} L_q(1/2, \chi)^2 = \sum_{\substack{F \in \mathcal{H}_{q^2,\, g/2 + 1} \\ P \mid F \Rightarrow P \notin \mathbb{F}_q[T]}} L_q(1/2, \chi_F)^2,
\]
where $P$ ranges over irreducible monic polynomials dividing $F$. Substituting the approximate functional equation \eqref{approx fe intro} into the right-hand side and interchanging sums, one has the following
\begin{equation}\label{setup eqn intro}
\sum_{\substack{F \in \mathcal{H}_{q^2,\, g/2 + 1} \\ P \mid F \Rightarrow P \notin \mathbb{F}_q[T]}} L_q(1/2, \chi_F)^2 
= \left(1 - \frac{1}{\sqrt{q}}\right)^2 \left( \sum_{i=0}^A (i+1) S_{A - i, \mathrm{prin}} + \sum_{i=0}^{2g - A - 1} (i+1) S_{2g - A - 1 - i, \mathrm{dual}} \right),
\end{equation}
where
\[
S_{t, \mathrm{prin}} := \sum_{f \in \cM_{\leq t}} \frac{d(f)}{\sqrt{|f|}} \left( \sum_{\substack{F \in \mathcal{H}_{q^2,\, g/2 + 1} \\ P \mid F \Rightarrow P \notin \mathbb{F}_q[T]}} \chi_F(f) \right)
\]
is the \emph{principal term} and
\[
S_{t, \mathrm{dual}} := \sum_{f \in \cM_{\leq 2g - t - 1}} \frac{d(f)}{\sqrt{|f|}} \left( \sum_{\substack{F \in \mathcal{H}_{q^2,\, g/2 + 1} \\ P \mid F \Rightarrow P \notin \mathbb{F}_q[T]}} \omega(\chi_F)^2\, \bar{\chi}_F(f) \right)
\]
is called the \emph{dual term}. Decompose the principal term as
\[
S_{t, \rm{prin}} = S_{t, \rm{prin}, \tinycube} + S_{t, \rm{prin}, \neq \tinycube},
\]
where the subscripts indicate the contribution from perfect cubes and non-cubes, respectively. The main term in our asymptotic arises from \( S_{t, \mathrm{prin}, \tinycube} \); see Section~\ref{main tem section}. The contributions from \( S_{t, \mathrm{prin}, \neq \tinycube} \) and the dual term \( S_{t, \mathrm{dual}} \) are shown to be of lower order and are absorbed into the error term which arises from \( S_{t, \mathrm{prin}, \tinycube} \). The analysis of \( S_{t, \mathrm{dual}} \) relies crucially on bounds for averages of Gauss sums (see Section~\ref{dual term section}). 

\subsection{Organization}
Including the introduction, the article consists of three sections. Section \ref{s 2} is devoted to establishing notation and developing background notions. We discuss zeta and L-functions associated to cubic Dirichlet characters. Proposition \ref{approx funct eqn} establishes the approximate functional equation \eqref{approx fe intro}. We recall key results on Gauss sums and their averages, see Proposition \ref{Gauss average Propn}. Theorem \ref{main theorem} is proven in section \ref{s 3}. In section \ref{s 3.1}, it is shown that the second moment decomposes as in \eqref{setup eqn intro}. The terms \( S_{t, \mathrm{prin}, \tinycube} \), \( S_{t, \mathrm{prin}, \neq \tinycube} \) and \( S_{t, \mathrm{dual}, \tinycube} \) are studied in sections \ref{main tem section}, \ref{noncube section} and \ref{dual term section} respectively.

\subsection{Acknowledgement} We would like to thank Alexandra Florea for encouraging comments and helpful feedback. 

\section{Preliminary notions}\label{s 2}
\par In this section, we discuss preliminary notions and set up notation used in the rest of the article.
\subsection{Basic notation}
\begin{itemize}
    \item Fix an odd prime number \( p \), and let \( q \) be a power of \( p \). Denote by \( \mathbb{F}_q \) the finite field with \( q \) elements.
    \item For any nonzero polynomial \( f \in \mathbb{F}_q[T] \), define its norm by \( |f| := q^{\deg(f)} \). Interpreting \( \F_q(T) \) as the field of functions of \( \mathbb{P}^1_{/\F_q} \), note that \( |\cdot| \) is the absolute value at \( \infty = (T^{-1}) \).
    \item A polynomial \( P \in \mathbb{F}_q[T] \) is said to be \emph{prime} if it is monic, irreducible, and nonconstant. The nonzero prime ideals of \( \mathbb{F}_q[T] \) are the principal ideals generated by prime polynomials. 
    \item There is an analog of the Möbius function over \( \mathbb{F}_q[T] \). If \( f \) is a nonzero squarefree polynomial with distinct prime factors \( P_1, \dots, P_r \), then \( \mu(f) := (-1)^r \). If \( f \) is not squarefree, then \( \mu(f) := 0 \). The function \( \mu \) plays an important role in arguments involving inclusion and exclusion.
    \item Given a nonzero polynomial \( f \in \F_q[T] \), $ \varphi(f)$ is the number of elements in $\left( \F_q[T]/(f) \right)^\times$.
    \item For \( k \geq 2 \), the \( k \)-th divisor function \( d_k(f) \) is given by $
    d_k(f) := \sum_{f_1 \cdots f_k = f} 1$,
    and set \( d(f) := d_2(f) \).
\end{itemize}

\subsection{The zeta function}
The zeta function of the affine line \( \mathbb{A}^1_{\mathbb{F}_q} \) is denoted $\zeta_q(s)$ and is defined for $\op{Re}(s)>1$ by the Dirichlet series $\zeta_q(s) := \sum_{f \in \mathcal{M}_q} \frac{1}{|f|^s}$. One finds that
\[
\zeta_q(s) = \sum_{n \geq 0} \left( \sum_{f \in \mathcal{M}_n} 1 \right) q^{-ns} = \sum_{n \geq 0} \frac{q^n}{q^{ns}} = \frac{1}{1 - q^{1-s}}.
\]
\noindent This function admits an Euler product decomposition of the form

\[
\zeta_q(s) = \prod_P \left( 1 - |P|^{-s} \right)^{-1},
\]
\noindent where the product runs over all prime polynomials \( P \). Introducing the variable \( u = q^{-s} \), we note that \( |u| = q^{-\operatorname{Re}(s)} < 1 \) in the domain of convergence and define 

\[
\mathcal{Z}(u) := \zeta_q(s) = \sum_{f \in \mathcal{M}} u^{\op{deg}(f)} = \frac{1}{1 - qu}.
\]
We recall a version of Perron's formula over \( \mathbb{F}_q[T] \), which will be used multiple times in this article. Associate to a function \( a: \mathcal{M} \to \mathbb{C} \), the generating series:
\[
\mathcal{A}(u) := \sum_{f \in \mathcal{M}} a(f) u^{\op{deg}(f)}.
\] 
\begin{theorem}\label{Perron's thm}
    With respect to notation above, assume that $\mathcal{A}(u)$ is absolutely convergent on the disk $|u|\leq r$, where $r\in (0, 1)$. Then for $n\geq 1$, the following assertions hold:
    \begin{enumerate}
        \item $\sum_{f \in \mathcal{M}_n} a(f)=\frac{1}{2 \pi i} \oint_{|u|=r} \frac{\mathcal{A}(u)d u}{u^{n+1}}$
        \item $\sum_{f \in \mathcal{M}_{\leq n}} a(f)=\frac{1}{2 \pi i} \oint_{|u|=r} \frac{\mathcal{A}(u)d u}{(1-u) u^{n+1}}$,
    \end{enumerate}
    where $\oint$ is the counterclockwise integral over the circle $|u|=r$.
\end{theorem}
\begin{proof}
    The proof follows from a direct application of Cauchy's integral formula.
\end{proof}
\subsection{Characters and L-polynomials}
\par A Dirichlet character modulo a polynomial \( f \) is a function \( \chi: \mathbb{F}_q[T] \to \mathbb{C} \) which is completely multiplicative for which \( \chi(g + fh) = \chi(g) \) for all polynomials \( g, h \in \mathbb{F}_q[T] \), with the additional requirement that \( \chi(h) \neq 0\) if and only if \( (h, f) = 1 \). Alternatively, we may view $\chi$ as a homomorphism $\left(\F_q[T]/(f)\right)^\times \rightarrow \mathbb{C}^\times$. Given two characters $\chi$ and $\psi$ modulo $(f)$, we set 
\[\delta(\chi, \psi):=\begin{cases}
    1 & \text{ if } \chi=\psi\\
    0 & \text{ otherwise}.
\end{cases}\]
One has the following orthogonality relation:
\begin{equation}\label{orthogonalityreln}
\sum_{g \bmod f} \chi(g) \overline{\psi(g)} = \varphi(f) \delta(\chi, \psi).
\end{equation}
For $g,f\in \left(\F_q[T]/(f)\right)^\times$, set 
\[\delta(f, g):=\begin{cases}
    1 & \text{ if } f=g;\\
    0 & \text{ otherwise.}
    \end{cases}\]
Summing over all Dirichlet characters modulo $(f)$, we have the relation:
\begin{equation}\label{orthogonality reln 2}
\sum_{\chi \bmod f} \chi(g) \overline{\chi(h)} = \varphi(f) \delta(g, h).
\end{equation}

\par Note that if $q\equiv 1\pmod{3}$, then, $\F_q^\times$ contains the cube roots of unity. We shall implicitly choose an isomorphism between the cube roots of unity in $\mathbb{C}^\times$ and those in $\F_q^\times$. Upon making this identification, we let $\chi_3$ be the cubic character on $\F_q^\times$ defined by $\chi_3(\alpha):=\alpha^{\frac{q-1}{3}}$. Let $\mathbf{1}$ be the trivial character. The characters of $\F_q^\times$ are $\mathbf{1}$, $\chi_3$ and $\chi_3^2$. Note that they are distinct if and only if $q\equiv 4,7\pmod{9}$. We say that a character $\chi$ on $\F_q[T]$ is even (resp. odd) if it is trivial (resp. nontrivial) on $\F_q^\times$. If $q\equiv 2\pmod{3}$, then any cubic character on $\F_q[T]$ is even. Define the \emph{$\mathcal{L}$-polynomial} associated to $\chi$ as follows:
\[\cL(u, \chi):=\sum_{f\in \cM} \chi(f) u^{\op{deg}(f)}=\sum_{n\geq 0} \left(\sum_{f\in \cM_n}\chi(f)\right)u^n.\]
Let $h$ be the conductor of $\chi$ and assume that $\chi$ is not trivial. It follows from the orthogonality relation \eqref{orthogonalityreln} that if $n\geq \op{deg}(h)$, then
\[\sum_{f\in \cM_n} \chi(f)=0\] and thus $\cL(u, \chi)$ is a polynomial of degree $\leq \op{deg}(h)-1$. Setting $u:=q^{-s}$ we write $L_q(s, \chi)=\cL(u, \chi)$. Let \( C \) be a curve over \( \mathbb{F}_q(T) \) whose function field defines a cyclic cubic extension of \( \mathbb{F}_q(T) \). By the Weil conjectures, the zeta function of \( C \) takes the form  
\[
\mathcal{Z}_{C}(u) = \frac{\mathcal{P}_{C}(u)}{(1 - u)(1 - q u)},
\]
\noindent where the numerator polynomial \( \mathcal{P}_{C}(u) \) is given by  
\[
\mathcal{P}_{C}(u) = \prod_{j=1}^{g} \left( 1 - \sqrt{q} u e^{2 \pi i \theta_j} \right) \prod_{j=1}^{g} \left( 1 - \sqrt{q} u e^{-2 \pi i \theta_j} \right),
\]
\noindent for some eigenangles \( \theta_j \), \( j = 1, \dots, g \). Since the function field of \( C \) corresponds to a cyclic cubic extension of \( \mathbb{F}_q(T) \), its arithmetic is naturally described in terms of the two associated cubic Dirichlet characters \( \chi \) and \( \bar{\chi} \). In particular, \( \mathcal{P}_C(u) \) can be expressed in terms of the \( L \)-functions of these characters. Setting

\[
\mathcal{L}_{C}(u, \chi) = 
\begin{cases} 
\mathcal{L}_{q}(u, \chi) & \text{if } \chi \text{ is odd}, \\  
\frac{\mathcal{L}_{q}(u, \chi)}{1 - u} & \text{if } \chi \text{ is even},  
\end{cases}
\]
\noindent one has the factorization:
\[
\mathcal{P}_{C}(u) = \mathcal{L}_{C}(u, \chi) \mathcal{L}_{C}(u, \bar{\chi}).
\]

\subsection{The approximate functional equation}
 Let $\chi$ be a nontrivial, primitive Dirichlet character of conductor $h$. Define the \emph{sign} of $\chi$ by
\[
\omega(\chi):= - q^{-\frac{(\op{deg}(h) -\delta)}{2}} \sum_{f \in \mathcal{M}_{\op{deg}(h) - 1}} \chi(f),
\]  
where \[\delta:=\begin{cases}
    1 & \text{ for }\chi \text{ odd},\\
    2 & \text{ for }\chi\text{ even}.\\
\end{cases}\]
According to the Riemann-Hurwitz formula, we have that
\begin{equation}\label{RH formula}
\op{deg}(h)=g+2- \begin{cases}0 & \text { if } \chi \text { is even,}\\ 1 & \text { if } \chi \text { is odd.}\end{cases}
\end{equation}
\noindent We state the functional equation for $\mathcal{L}_{q}(u, \chi)$. A natural dichotomy arises, depending on whether $\chi$ is odd or even.
\begin{lemma} \label{functional equation lemma}
With respect to notation above, the following assertions hold.
\begin{enumerate}
    \item If $\chi$ is odd, we have that
    \begin{equation}\label{FE odd case}
\mathcal{L}_{q}(u, \chi) = \omega(\chi) (\sqrt{q} u)^{g} \mathcal{L}_{q} \left(\frac{1}{q u}, \bar{\chi} \right).
\end{equation}
\item On the other hand, suppose that $\chi$ is even, then 
\begin{equation}\label{FE even case}
\mathcal{L}_{q}(u, \chi) = \omega(\chi) (\sqrt{q} u)^{g} \frac{1 - u}{1 - \frac{1}{q u}} \mathcal{L}_{q} \left(\frac{1}{q u}, \bar{\chi} \right).
\end{equation}
\end{enumerate}
\end{lemma}

\begin{proof}
    This result is \cite[Lemma 2.3]{DFLCubic}
\end{proof}

\noindent The functional equation gives rise to a symmetry in the polynomial expression for $\cL_q(u, \chi)^k$, which turn gives a nice expression for $L(1/2, \chi)^k$, commonly known as the \emph{approximate functional equation}.

\begin{proposition}[Approximate functional equation]\label{approx funct eqn}
    Let $\chi$ be a nontrivial primitive cubic character of modulus $h$. If $\chi$ is odd, then
    \[L(1/2, \chi)^k=\sum_{f\in \cM_{\leq A}}\frac{\chi(f)d_k(f)}{\sqrt{|f|}}+ \omega(\chi)^k\sum_{f\in \cM_{\leq kg-A-1}}  \frac{\bar{\chi}(f) d_{k}(f)}{\sqrt{|f|}}.\]
    On the other hand, if $\chi$ is even, 
    \[\begin{split}L(1/2, \chi)^k=&(1-q^{-1/2})^k \sum_{i=0}^{A}\binom{k+i-1}{i}\sum_{f\in \cM_{\leq A-i}}\frac{\chi(f)d_k(f)}{\sqrt{|f|}} \\ + & (1-q^{-1/2})^k \omega(\chi)^k \sum_{i=0}^{kg-A-1}\binom{k+i-1}{i}\sum_{f\in \cM_{\leq kg-A-1-i}}  \frac{\bar{\chi}(f) d_{k}(f)}{\sqrt{|f|}}.
    \end{split}\]
\end{proposition}

\begin{proof}
First we assume that $\chi$ is odd, then by part (1) of Lemma \ref{functional equation lemma}, we have that
$$
\mathcal{L}_{q}(u, \chi)^k=\omega(\chi)^k(\sqrt{q} u)^{kg} \mathcal{L}_{q}\left(\frac{1}{q u}, \bar{\chi}\right)^k .
$$  Writing
$
\mathcal{L}_q\left(u, \chi\right)^{k}=\sum_{n=0}^{k g} b_n(\chi) u^{n}
$,
where $b_n(\chi)=\sum_{f \in \mathcal{M}_{n}} \chi(f) d_{k}(f)$, it follows that
$$
\sum_{n=0}^{k g} b_n(\chi) u^{n}=\omega(\chi)^k (\sqrt{q} u)^{kg} \sum_{n=0}^{k g} b_n(\bar{\chi}) \left(\frac{1}{qu}\right)^n
=\sum_{n=0}^{kg} b_n(\bar{\chi}) \omega(\chi)^k q^{kg/2-n} u^{kg-n}.$$
Therefore, we have that \[b_n(\chi)=b_{kg-n}(\bar{\chi})\omega(\chi)^k q^{n-kg/2},\] i.e., 
\[\sum_{f \in \mathcal{M}_{n}} \chi(f) d_{k}(f)=\sum_{f \in \mathcal{M}_{kg-n}} \omega(\chi)^k q^{n-kg/2}\bar{\chi}(f) d_{k}(f).\]
We find that:
\[\begin{split}\mathcal{L}_{q}(u, \chi)^k=& \sum_{n=0}^{A} \left(\sum_{f \in \mathcal{M}_{n}} \chi(f)d_k(f)\right) u^n+\sum_{n=A+1}^{kg} \left(\sum_{f \in \mathcal{M}_{n}} \chi(f)d_k(f)\right) u^n\\
= & \sum_{n=0}^{A} \left(\sum_{f \in \mathcal{M}_{n}} \chi(f)d_k(f)\right) u^n+ \sum_{n=A+1}^{kg} \left(\sum_{f \in \mathcal{M}_{kg-n}} \omega(\chi)^k q^{n-kg/2}\bar{\chi}(f) d_{k}(f)\right)u^n \\
= & \sum_{n=0}^{A} \left(\sum_{f \in \mathcal{M}_{n}} \chi(f)d_k(f)\right) u^n+ q^{kg/2}u^{kg}\omega(\chi)^k\sum_{n=0}^{kg-A-1} \left(\sum_{f \in \mathcal{M}_{n}} \bar{\chi}(f)  d_{k}(f)\right)(qu)^{-n}.
\end{split}\]
Plug in $u=q^{-1/2}$ to the above expression to get that 
\[L(1/2, \chi)^k=\sum_{f\in \cM_{\leq A}}\frac{\chi(f)d_k(f)}{\sqrt{|f|}}+ \omega(\chi)^k\sum_{f\in \cM_{\leq kg-A-1}}  \frac{\bar{\chi}(f) d_{k}(f)}{\sqrt{|f|}}.\] This proves part (1) of the Proposition.
\par Next we prove part (2) and assume that $\chi$ is even. In this case, we write
$$
F_\chi(u):=\left(\frac{\mathcal{L}_q\left(u, \chi\right)}{(1-u)}\right)^{k}=\sum_{n=0}^{k g} b_n(\chi) u^{n}.
$$
By the functional equation in part (2) of Lemma \ref{functional equation lemma}, 
\[F_\chi(u)=\omega(\chi)^k (\sqrt{q} u)^{kg}F_{\bar{\chi}}\left(\frac{1}{qu}\right)\]
and thus, $b_n(\chi)=b_{kg-n}(\bar{\chi})\omega(\chi)^k q^{n-kg/2}$.
It follows that 
\begin{align*}
    F_\chi(u)&= \sum_{n=0}^{A} b_n(\chi) u^{n}+\sum_{n=A+1}^{kg} b_n(\chi) u^{n}\\
    &=\sum_{n=0}^{A} b_n(\chi) u^{n}+\sum_{n=A+1}^{kg}b_{kg-n}(\bar{\chi})\omega(\chi)^k q^{n-kg/2}u^n\\
    &=\sum_{n=0}^{A} b_n(\chi) u^{n}+\omega(\chi)^k q^{kg/2}u^{kg}\sum_{n=0}^{kg-A-1}b_n(\bar \chi)(qu)^{-n}.\numberthis\label{functional equation}
\end{align*}
Now, if we  write
$$
\mathcal{L}_q\left(u, \chi\right)^{k}=\sum_{n=0}^{k g+k} a_n(\chi) u^{n}
$$
where $a_n(\chi)=\sum_{f \in \mathcal{M}_{n}} \chi(f) d_{k}(f)$.  
Note that $(1-u)^{-k}\mathcal{L}_q\left(u, \chi\right)^{k}=F_\chi(u)$. Using the Taylor expansion for $(1-u)^{-k}$ gives us the following relation: 
\begin{equation}\label{binomial inversion}
    b_n(\chi)=\sum_{i=0}^{n} \binom{k+i-1}{i} a_{n-i}(\chi).
\end{equation}
 Noting that $\mathcal{L}_q\left(u, \chi\right)^{k}=(1-u)^k\sum_{n=0}^{kg}b_n u^n$, we deduce from the functional equation \eqref{functional equation}:
\begin{align*}
    \mathcal{L}_q\left(u, \chi\right)^{k}&=(1-u)^k\left(\sum_{n=0}^{A} b_n(\chi) u^{n}+\omega(\chi)^k q^{kg/2}u^{kg}\sum_{n=0}^{kg-A-1}b_n(\bar \chi)(qu)^{-n}\right).
\end{align*}
Setting $u=q^{-1/2}$, we arrive at the following:
{\small\begin{align*}&\mathcal{L}_q\left(1/\sqrt{q}, \chi\right)^{k}\\
    & =(1-q^{-1/2})^k\left(\sum_{n=0}^{A} b_n(\chi) q^{-n/2}+\omega(\chi)^k \sum_{n=0}^{kg-A-1}b_n(\bar \chi)q^{-n/2}\right)\\&=(1-q^{-1/2})^k\sum_{i=0}^{A}\binom{k+i-1}{i}\left(\sum_{n=i}^{A} a_{n-i}(\chi) q^{-n/2}+\omega(\chi)^k \sum_{n=i}^{kg-A-1}a_{n-i}(\bar \chi)q^{-n/2}\right)\\&
 =  (1-q^{-1/2})^k \left(\sum_{i=0}^{A}\binom{k+i-1}{i}\sum_{f\in \cM_{\leq A-i}}\frac{\chi(f)d_k(f)}{\sqrt{|f|}}+ \sum_{i=0}^{kg-A-1}\binom{k+i-1}{i}\omega(\chi)^k\sum_{f\in \cM_{\leq kg-A-1-i}}  \frac{\bar{\chi}(f) d_{k}(f)}{\sqrt{|f|}}\right).
\end{align*}}
\end{proof}
\par For any odd character \(\chi\) on \(\mathbb{F}_q[T]\), let \(\tau(\chi)\) be the Gauss sum of its restriction to \(\mathbb{F}_q^\times\), given by  
\[
\tau(\chi) = \sum_{a \in \mathbb{F}_{q}^{\times}} \chi(a) \op{exp}\left(\frac{2\pi i \op{tr}_{\mathbb{F}_{q} / \mathbb{F}_{p}}(a)}{ p}\right),
\]
where $\op{tr}_{\F_q/\F_p}: \F_q\rightarrow \F_p$ is the trace map. This satisfies \(|\tau(\chi)| = q^{1/2}\), and we define the associated sign  
\[
\varepsilon(\chi) = q^{-1/2} \tau(\chi).
\] On the other hand, if $\chi$ is even, simply set $\tau(\chi):=1$. Given $a\in \F_q\left((\frac{1}{T})\right)$, let $a_1$ be the coefficient of $\frac{1}{T}$. Hayes defined an exponential function $e_q: \F_q\left((\frac{1}{T})\right)\rightarrow  \mathbb{C}$ as follows:
\[e_q(a):=\op{exp}\left(\frac{2\pi i  \op{tr}_{\F_q/\F_p} (a_1)}{p}\right).\] It is easy to check that $e_q(\cdot)$ has the following properties:
\begin{itemize}
    \item $e_q(a)=1$ for $a\in \F_q[T]$, 
    \item $e_q(a+b)=e_q(a)e_q(b)$ for all $a, b\in \F_q\left(\frac{1}{T}\right)$,
    \item if $a, b, h\in \F_q[T]$ are such that $a\equiv b\pmod{h}$, then $e_q(a/h)=e_q(b/h)$.
\end{itemize}
For a primitive character \(\chi\) of modulus \(h\) on \(\mathbb{F}_q[T]\), the Gauss sum is defined as  
\[
G(\chi) = \sum_{a \bmod h} \chi(a) e_q \left(\frac{a}{h} \right).
\]  
The root number is expressed in terms of Gauss sums as follows:  
\begin{equation}\label{omega function}
    \omega(\chi) =  
\begin{cases}  
\dfrac{1}{\tau(\chi)} q^{-\frac{\op{deg}(h) - 1}{2}} G(\chi) & \text{if } \chi \text{ is odd}, \\  \\
\dfrac{1}{\sqrt{q}} q^{-\frac{\op{deg}(h) - 1}{2}} G(\chi) & \text{if } \chi \text{ is even}, 
\end{cases}  
\end{equation}
cf. \cite[Corollary 2.3]{DFLCubic}.

\subsection{Primitive characters of $\F_q[T]$}\label{section 2.5}
\par We introduce notation for cubic characters. Let \( q \) be an odd prime power with \( q \equiv 1 \mod 3 \), and let \( P \) be a prime polynomial in \( \mathbb{F}_q[T] \). The cubic residue symbol \( \chi_P \) is defined for \( a \in \mathbb{F}_q[T] \) by  
\[
\chi_P(a) = \begin{cases}  
0, & P \mid a, \\  
\alpha, & \text{otherwise},  
\end{cases}  
\]
where \( \alpha \) is the unique root of unity in \( \mathbb{C} \) satisfying  
\[
a^{(q^{\deg P} - 1)/3} \equiv \Omega(\alpha )\mod P.
\]  
For a monic polynomial \( G = P_1^{e_1} \cdots P_s^{e_s} \) with distinct irreducible factors \( P_i \), we set  
\[
\chi_G = \chi_{P_1}^{e_1} \cdots \chi_{P_s}^{e_s},
\]
a cubic character modulo \( P_1 \cdots P_s \). It is primitive if each \( e_i \) is \( 1 \) or \( 2 \), in which case its conductor is the square-free polynomial \( F = P_1 \cdots P_s \). 
\par Next, let \( q \) be an odd prime power with \( q \equiv 2 \mod 3 \), then there is no primitive cubic character if \( P \) be an irreducible monic polynomial  of odd dgree in \( \mathbb{F}_q[T] \) since \(3\nmid q^{\op{deg}(P)}-1\). For even degree \( P \)  irreducible monic polynomial and \( a \in \mathbb{F}_q[T] \), the cubic residue symbol  is  
\(\chi_P(a) = \alpha \), where \( \alpha \) is the unique root of unity in \( \mathbb{C} \) satisfying  
\[
a^{(q^{\deg P} - 1)/3} \equiv \Omega(\alpha) \mod P.
\] 
Here, $\Omega$ takes values in the cubic roots of unity in \(\mathbb{F}_{q^2}\). For a monic polynomial \( G = P_1^{e_1} \cdots P_s^{e_s} \) with distinct irreducible factors \( P_i \) of even degree, we set  
\[
\chi_G = \chi_{P_1}^{e_1} \cdots \chi_{P_s}^{e_s},
\]
a cubic character modulo \( P_1 \cdots P_s \). It is primitive if each \( e_i \) is \( 1 \) or \( 2 \). Its conductor is the square-free polynomial \( F=  P_1 \cdots P_s \). Consider the Galois extension $\F_{q^2}(T)/\F_q(T)$ with $\op{Gal}\left(\F_{q^2}(T)/\F_q(T)\right)=\langle \sigma\rangle$ where \begin{equation}\label{sigma defn}\sigma\left(\frac{\sum_{i} a_i T^i}{\sum_{i} b_i T^i}\right):=\frac{\sum_{i} a_i^q T^i}{\sum_{i} b_i^q T^i}.\end{equation} Given $f\in \F_q$ we set $\tilde{f}:=\sigma(f)$. Let $P$ be a prime of $\F_q[T]$. Note that $P$ splits in $\F_{q^2}[T]$ if and only if $\op{d}(P)$ is even, in which case we write $P=\pi \widetilde{\pi}$, where $\pi$ is a prime of $\F_{q^2}[T]$. It is easy to see that the restriction of $\chi_{\pi}$ (resp. $\chi_{\tilde{\pi}}$) to $\F_q[T]$ is $\chi_P$ (resp. $\overline{\chi_P}$). Then by running over all the characters $\chi_{F}$ where $F \in \mathbb{F}_{q^{2}}[T]$ is square-free and not divisible by a prime $P$ of $\mathbb{F}_{q}[T]$, we are counting exactly the characters over $\mathbb{F}_{q^{2}}[T]$ whose restrictions are cubic characters over $\mathbb{F}_{q}[T]$, and each character over $\mathbb{F}_{q}[T]$ is counted exactly once.
\par When $q\equiv 2\mod{3}$, any cubic character of $\F_q^\times$ is trivial, and hence, any cubic character of $\F_q[T]$ is even.
By \eqref{RH formula}, if $F \in \mathbb{F}_{q}[T]$ is the conductor of a cubic primitive character $\chi$ over $\mathbb{F}_{q}[T]$, it follows that $\operatorname{deg}(F)=g+2$. By the classification above, it follows that $F=P_{1} \ldots P_{s}$ for distinct primes of even degree, and the character $(\bmod F)$ is the restriction of a character
of conductor $\pi_{1} \ldots \pi_{s}$ over $\mathbb{F}_{q^{2}}[T]$, where $\pi_{i}$ is one of the primes lying above $P_{i}$. Then the degree of the conductor of this character over $\mathbb{F}_{q^{2}}[T]$ is equal to $g / 2+1$. In what follows, $P$ will denote a prime polynomial of $\F_q[T]$. We have the following result.
\begin{lemma}\label{cubic sum q 2 mod 3}
For $q \equiv 2 \pmod{3}$, we have that
$$
\sum_{\substack{\chi \text { primitive cubic } \\ \text { genus }(\chi)=g}} L_{q}(1 / 2, \chi)^k=\sum_{\substack{F \in \mathcal{H}_{q^{2}, g / 2+1} \\ P \mid F \Rightarrow P \notin \mathbb{F}_{q}[T]}} L_{q}\left(1 / 2, \chi_{F}\right)^k,
$$
where $\chi_F$ is defined over $\F_{q^2}[T]$. 
\end{lemma}

We conclude this subsection with an important fact about the values of $\chi_F$ when restricted to $\F_q[T]$. This result will be invoked frequently in the arguments that follow.

\begin{lemma}\label{chi_D(f)=1 lemma}
    Suppose that $q\equiv 2\pmod{3}$, $F$ be a squarefree polynomial in $\F_q[T]$ and $\chi_F$ is the character associated to $F$ over $\F_{q^2}[T]$. Then, for any squarefree $f\in \F_q[T]$, we have that $\chi_F(f)=1$.
\end{lemma}

\begin{proof}
\par Let $P$ be a prime polynomial in $\F_q[T]$ dividing $F$. Since $\chi_F = \prod_{P \mid F} \chi_P$, it suffices to show that $\chi_P(f) = 1$ for each such $P$.
\par First, suppose that $\deg P$ is odd. In this case, the prime ideal $(P)$ remains inert in $\F_{q^2}[T]$. By definition, $\chi_P(f) = \alpha$, where $\alpha \in \mathbb{C}$ is the unique root of unity satisfying
$$
f^{(q^{2\deg P} - 1)/3} \equiv \Omega(\alpha) \pmod{P}.
$$
\noindent The element $f^{(q^{2\deg P} - 1)/3}$ lies in $\left( \F_q[T]/(P) \right)^\times$, and is a cube root of unity. But since $\deg P$ is odd, $\F_q[T]/(P) \cong \F_{q^{\deg P}}$, and this field contains no nontrivial cube roots of unity when $q \equiv 2 \pmod{3}$. It follows that $\alpha = 1$, and hence $\chi_P(f) = 1$.

\par Now suppose that $\deg P$ is even. Then $P$ splits in $\F_{q^2}[T]$ as a product of two distinct monic irreducibles $Q$ and $\widetilde{Q}$, with $\widetilde{Q} = \sigma(Q)$, where $\sigma$ denotes the generator of $\op{Gal}(\F_{q^2}(T)/\F_q(T))$ as defined in \eqref{sigma defn}. We then have $\chi_P = \chi_Q \chi_{\widetilde{Q}}$. Moreover, $\sigma$ induces an isomorphism
$$
\sigma : \left( \F_{q^2}[T]/(Q) \right)^\times \xrightarrow{\sim} \left( \F_{q^2}[T]/(\widetilde{Q}) \right)^\times,
$$
\noindent which implies that $\chi_{\widetilde{Q}}(f) = \chi_Q(f)^q$. Since $\chi_Q(f)$ is a cube root of unity and $q \equiv 2 \pmod{3}$, we conclude that $\chi_{\widetilde{Q}}(f) = \chi_Q(f)^{-1}$. Thus,
$$
\chi_P(f) = \chi_Q(f) \chi_{\widetilde{Q}}(f) = 1,
$$
as desired.

\end{proof}

\subsection{Averages of Gauss sums}
\par For $q\equiv 1\pmod{6}$, let $f\in \F_q[T]$ and $\chi_f$ the associated cubic character. Define the generalized Gauss sum as follows
\begin{equation*}
G_{q}(V, f)=\sum_{u(\bmod f)} \chi_{f}(u) e_{q}\left(\frac{u V}{f}\right).
\end{equation*}
If \( (f_1, f_2) = 1 \), then  
\[
G_q(V, f_1 f_2) = \chi_{f_1}(f_2)^2 G_q(V, f_1) G_q(V, f_2) = G_q(V f_2, f_1) G_q(V, f_2).
\]  
If \( V = V_1 P^\alpha \) where \( P \nmid V_1 \), then  
\[
G_q(V, P^i) =  
\begin{cases}  
0 & \text{if } i \leq \alpha \text{ and } i \not\equiv 0 \pmod{3}, \\  
\varphi(P^i) & \text{if } i \leq \alpha \text{ and } i \equiv 0 \pmod{3}, \\  
-|P|_q^{i-1} & \text{if } i = \alpha+1 \text{ and } i \equiv 0 \pmod{3}, \\  
\varepsilon(\chi_{P^i}) \omega(\chi_{P^i}) \chi_{P^i}(V_1^{-1}) |P|_q^{i-\frac{1}{2}} & \text{if } i = \alpha+1 \text{ and } i \not\equiv 0 \pmod{3}, \\  
0 & \text{if } i \geq \alpha+2.
\end{cases}  
\]  
\noindent We recall that \( \varepsilon(\chi) = 1 \) when \( \chi \) is even. For the case of \( \chi_{P^i} \), this happens if \( 3 \mid \deg(P^i) \). Moreover, if $(a, f)=1$, it is easy to see that
\begin{equation}\label{gauss sum relation}
G_{q}(a V, f)=\overline{\chi_f}(a) G_{q}(V, f).
\end{equation}
For further details, we refer to \cite[Lemma 2.12]{DFLCubic}.
\begin{remark}
    Note that the Gauss sum $G_{q^2}(V, f)$ is defined for odd $q$ with $q \equiv 2 \pmod{3}$, and we shall make use of its properties in what follows. 
\end{remark}

\begin{proposition}
     Let $f\in \mathbb{F}_{q}[T]$ be a monic polynomial and $n:=\op{deg}(f)$. Then the following assertions hold.
     \begin{enumerate}
         \item Suppose that $n\equiv 0\pmod{3}$. Then, $$
\sum_{h \in \mathcal{M}_{q, m}} \chi_{f}(h)=\frac{q^{m}}{|f|_{q}}\left[G_{q}(0, f)+(q-1) \sum_{V \in \mathcal{M}_{q, \leq n-m-2}} G_{q}(V, f)-\sum_{V \in \mathcal{M}_{q, n-m-1}} G_{q}(V, f)\right] .
$$
\item Suppose that $n\not \equiv 0\pmod{3}$. Then we have that $$
\sum_{h \in \mathcal{M}_{q, m}} \chi_{f}(h)=\frac{q^{m+\frac{1}{2}}}{|f|_{q}} \overline{\varepsilon\left(\chi_{f}\right)} \sum_{V \in \mathcal{M}_{q, n-m-1}} G_{q}(V, f).
$$
     \end{enumerate}

\end{proposition}

\begin{proof}
    The result above is \cite[Proposition 2.13]{DFLCubic}.
\end{proof}

Next we state results regarding the averages of Gauss sums. Let

\[\Psi_{q}(f, u)=\sum_{F \in \mathcal{M}_{q}} G_{q}(f, F) u^{\op{deg}(F)}\]
and
\begin{equation*}
\tilde{\Psi}_{q}(f, u)=\sum_{\substack{F \in \mathcal{M}_{q} \\(F, f)=1}} G_{q}(f, F) u^{\op{deg}(F)}.
\end{equation*}

For $a, n \in \mathbb{Z}$ and $n$ positive, we denote by $[a]_{n}$ the residue of $a$ modulo $n$ such that $0 \leq[a]_{n} \leq n-1$.
\begin{proposition}\label{Gauss average Propn}
    Assume that $q\equiv 1\pmod{6}$. Let $f=f_{1} f_{2}^{2} f_{3}^{3}$ with $f_{1}$ and $f_{2}$ square-free and coprime. We have
$$
\begin{aligned}
\sum_{\substack{F \in \mathcal{M} q, d \\
(F, f)=1}} G_{q}(f, F)= & \delta_{f_{2}=1} \frac{q^{\frac{4 d}{3}-\frac{4}{3}\left[d+\operatorname{deg}\left(f_{1}\right)\right]_{3}}}{\zeta_{q}(2)\left|f_{1}\right|_{q}^{2 / 3}} \overline{G_{q}\left(1, f_{1}\right)} \rho\left(1,\left[d+\operatorname{deg}\left(f_{1}\right)\right]_{3}\right) \prod_{P \mid f_{1} f_{3}^{*}}\left(1+\frac{1}{|P|_{q}}\right)^{-1} \\
& +O\left(\delta_{f_{2}=1} \frac{q^{\frac{d}{3}}+\varepsilon d}{\left|f_{1}\right|_{q}^{\frac{1}{6}}}\right)+\frac{1}{2 \pi i} \oint_{|u|=q^{-\sigma}} \frac{\tilde{\Psi}_{q}(f, u)}{u^{d}} \frac{d u}{u}
\end{aligned}
$$
with $2 / 3<\sigma<4 / 3$ and $\rho\left(1,\left[d+\operatorname{deg}\left(f_{1}\right)\right]_{3}\right)$ is given by \cite[(28)]{DFLCubic}. Moreover, we have
$$
\frac{1}{2 \pi i} \oint_{|u|=q^{-\sigma}} \frac{\tilde{\Psi}_{q}(f, u)}{u^{d}} \frac{d u}{u} \ll q^{\sigma d}|f|_{q}^{\frac{1}{2}\left(\frac{3}{2}-\sigma\right)}
$$
\end{proposition}
\begin{proof}
    This result is \cite[Proposition 3.1]{DFLCubic}.
\end{proof}

\section{Quadratic moments of cubic L-functions}\label{s 3}

\subsection{Setting up the problem}\label{s 3.1}
\par We now assume that \( q \) is odd and satisfies \( q \equiv 2 \bmod{3} \). In particular, the group \( \F_q^\times \) does not contain the third roots of unity, and consequently, every primitive cubic character over \( \F_q[T] \) is even. Recalling Lemma \ref{cubic sum q 2 mod 3}, we obtain the identity
\[
\sum_{\substack{\chi \text{ primitive cubic} \\ \text{genus}(\chi) = g}} L_q(1/2, \chi)^k 
= \sum_{\substack{F \in \mathcal{H}_{q^2,\, g/2 + 1} \\ P \mid F \Rightarrow P \notin \F_q[T]}} L_q(1/2, \chi_F)^k.
\]
\noindent From Proposition \ref{approx funct eqn}, the central value of the \( L \)-function raised to the \( k \)-th can be expressed as follows
\begin{align*}
    L(1/2, \chi)^k &= (1 - q^{-1/2})^k \left( \sum_{i=0}^{A} \binom{k+i-1}{i} \sum_{f \in \cM_{\leq A - i}} \frac{\chi(f) d_k(f)}{\sqrt{|f|}} \right.\\&+ \left.\sum_{i=0}^{kg - A - 1} \binom{k+i-1}{i} \omega(\chi)^k \sum_{f \in \cM_{\leq kg - A - 1 - i}} \frac{\bar{\chi}(f) d_k(f)}{\sqrt{|f|}} \right),
\end{align*}
\noindent where the value of $A$ will be chosen at the end of our analysis. Specializing the above equation to $k=2$, we deduce that
{\small \[\begin{split}L(1/2, \chi)^2=& \left(1-\frac{1}{\sqrt{q}}\right)^2\left(\sum_{i=0}^{A}(i+1)\sum_{f\in \cM_{\leq A-i}}\frac{\chi(f)d(f)}{\sqrt{|f|}}+ \sum_{i=0}^{2g-A-1}(i+1)\omega(\chi)^2\sum_{f\in \cM_{\leq 2g-A-1-i}}  \frac{\bar{\chi}(f) d(f)}{\sqrt{|f|}}\right).
\end{split}\]}
\noindent Thus we may write 
\begin{equation}\label{princ+dual}
    \sum_{\substack{F \in \mathcal{H}_{q^{2}, g / 2+1} \\ P \mid F \Rightarrow P \notin \mathbb{F}_{q}[T]}} L_{q}\left(1 / 2, \chi_{F}\right)^2 \\
  = \left(1-\frac{1}{\sqrt{q}}\right)^2\left(\sum_{i=0}^{A}(i+1)S_{A-i, \rm{prin}}+\sum_{i=0}^{2g-A-1}(i+1)S_{2g-A-1+i, \rm{dual}}\right), 
\end{equation}
 where 
{\small \begin{equation}\label{princpal eqn}S_{t, \rm{prin}}:=\sum_{f\in \cM_{\leq t}}\frac{d(f)}{\sqrt{|f|}}\left(\sum_{\substack{F \in \mathcal{H}_{q^{2}, g / 2+1} \\ P \mid F \Rightarrow P \notin \mathbb{F}_{q}[T]}} \chi_F(f)\right)
\quad\text{and} \quad
S_{t, \rm{dual}}:=\sum_{f\in \cM_{\leq 2g-t-1}}  \frac{d(f)}{\sqrt{|f|}}\left(\sum_{\substack{F \in \mathcal{H}_{q^{2}, g / 2+1} \\ P \mid F \Rightarrow P \notin \mathbb{F}_{q}[T]}} \omega(\chi_F)^2\bar{\chi}_F(f)\right).\end{equation}}
\begin{remark}Following \cite{Florea2ndand3rd}, the terms $S_{t, \rm{prin}}$ (resp. $S_{t, \rm{dual}}$) are called \emph{principal terms} (resp. \emph{dual terms}). One key difference with \emph{loc. cit.} is that in our work there are many principal (respectively, dual) terms, one for each $t = 0, \dots, A$ (respectively, $t = 0, \dots, 2g - A - 1$). This occurs because, in the functional equation for even cubic Dirichlet characters, there is an additional factor $\frac{1 - u}{1 - \frac{1}{q u}}$ appearing in the functional equation \eqref{FE even case}.
\end{remark}
 We decompose $S_{t, \mathrm{prin}} = S_{t, \mathrm{prin}, \tinycube} + S_{t, \mathrm{prin}, \neq \tinycube},
$ where the subscript $\tinycube$ (respectively, $\neq \tinycube$) indicates the sum over all $f$ that are perfect cubes (respectively, not perfect cubes). We prove estimates for the terms $S_{t, \mathrm{prin}, \tinycube}$ and $S_{t, \mathrm{prin}, \neq \tinycube}$ separately.
\subsection{The main term}\label{main tem section}
\par The main term in \eqref{eqn for main theorem} will arise from $S_{t, \rm{prin}, \tinycube}$. Express $t=3t_0+t_1$ where $t_0:=\lfloor t/3\rfloor$ and $t_1$ is $3\{t/3\}$. Writing $f=l^3$, one finds that 
\[S_{t, \rm{prin}, \tinycube}=\sum_{l\in \cM_{\leq t_0}}\frac{d(l^3)}{\sqrt{|l|^3}}\left(\sum_{\substack{F \in \mathcal{H}_{q^{2}, g / 2+1} \\ P \mid F \Rightarrow P \notin \mathbb{F}_{q}[T]}} 1\right).\numberthis\label{S-princ}\] 

\noindent First we estimate the inner sum over $F$, namely: \[\left(\sum_{\substack{F \in \mathcal{H}_{q^{2}, g / 2+1} \\ P \mid F \Rightarrow P \notin \mathbb{F}_{q}[T]}} 1\right).\] \noindent In order to handle the condition "$P \mid F \Rightarrow P \notin \mathbb{F}_q[T]$", we make use of the M\"obius function, specifically the identity
\begin{equation}\label{mobius identity}
\sum_{\substack{D \in \mathbb{F}_q[T] \\ D \mid F}} \mu(D) = 
\begin{cases}
1 & \text{if } F \text{ has no prime divisor in } \mathbb{F}_q[T], \\
0 & \text{otherwise}.
\end{cases}
\end{equation}
\noindent Consider the generating series:
\begin{equation*}
\sum_{\substack{F \in \mathcal{H}_{q^{2}} \\(F, l)=1 \\ P \mid F \Rightarrow P \notin \mathbb{F}_{q}[T]}} z^{\op{deg}(F)}=\sum_{\substack{F \in \mathcal{H}_{q^{2}} \\(F, l)=1}} z^{\op{deg}(F)} \sum_{\substack{D \in \mathbb{F}_{q}[T] \\ D \mid F}} \mu(D)=\sum_{\substack{D \in \mathbb{F}_{q}[T] \\(D, l)=1}} \mu(D) z^{\op{deg}(D)} \sum_{\substack{F \in \mathcal{H}_{q^{2}} \\(F, D l)=1}} z^{\op{deg}(F)} . 
\end{equation*}
\noindent We find that 
\[\sum_{\substack{F \in \mathcal{H}_{q^{2}} \\(F, l D)=1}} z^{\op{deg}(F)}=\prod_{\substack{P \in \mathbb{F}_{q^{2}}[T] \\ P \nmid D l}}\left(1+z^{\op{deg}(P)}\right)=\frac{\mathcal{Z}_{q^{2}}(z)}{\mathcal{Z}_{q^{2}}\left(z^{2}\right) \prod_{\substack{P \in \mathbb{F}_{q^{2}}[T] \\ P \mid D l}}\left(1+z^{\op{deg}(P)}\right)}\]
\noindent It follows that
\begin{equation*}
\sum_{\substack{F \in \mathcal{H}_{q^{2}} \\(F, l)=1 \\ P \mid F \Rightarrow P \notin \mathbb{F}_{q}[T]}} z^{\op{deg}(F)}=\frac{\mathcal{Z}_{q^{2}}(z)}{\mathcal{Z}_{q^{2}}\left(z^{2}\right) \prod_{\substack{P \in \mathbb{F}_{q^{2}}[T] \\ P \mid  l}}\left(1+z^{\operatorname{deg}(P)}\right)}\sum_{\substack{D \in \mathbb{F}_{q}[T] \\(D, l)=1}} \frac{\mu(D) z^{\op{deg}(D)}}{\prod_{\substack{P \in \mathbb{F}_{q^{2}}[T] \\ P \mid  D}}\left(1+z^{\op{deg}(P)}\right)}  . 
\end{equation*}
\noindent Consider the inner sum  of the above generating series, we have
\[\begin{split}
   \sum_{\substack{D \in \mathbb{F}_{q}[T] \\(D, l)=1}} \frac{\mu(D) z^{\op{deg}(D)}}{\prod_{\substack{P \in \mathbb{F}_{q^{2}}[T] \\ P \mid  D}}\left(1+z^{\operatorname{deg}(P)}\right)}&=\prod_{\substack{R\in \mathbb{F}_{q}[T] \\(R, l)=1 \\ \op{deg}(R)\ \text{odd}}}\left(1-\frac{z^{\deg(R)}}{1+z^{\deg(R)}}\right)\prod_{\substack{R\in \mathbb{F}_{q}[T] \\(R, l)=1 \\ \deg(R)\ \text{even}}}\left(1-\frac{z^{\deg(R)}}{(1+z^{\deg(R)/2})^2}\right)\\
   &=\prod_{\substack{R\in \mathbb{F}_{q}[T] \\(R, l)=1 \\ \deg(R)\ \text{odd}}}\left(\frac{1}{1+z^{\deg(R)}}\right)\prod_{\substack{R\in \mathbb{F}_{q}[T] \\(R, l)=1 \\ \deg(R)\ \text{even}}}\left(\frac{1+2z^{\deg(R)/2}}{(1+z^{\deg(R)/2})^2}\right)\\
&=\prod_{\substack{R\in \mathbb{F}_{q}[T] \\(R, l)=1 \\ \deg(R)\ \text{odd}}}A_R(z)\prod_{\substack{R\in \mathbb{F}_{q}[T] \\(R, l)=1 \\ \deg(R)\ \text{even}}}B_R(z)
 \end{split}\]
where  $R$ is a monic, irreducible polynomial, $A_R(z)=\frac{1}{1+z^{\deg(R)}}$, and $B_R(z)=\frac{1+2z^{\deg(R)/2}}{((1+z^{\deg(R)/2})^2}$
\noindent Putting everything together, we get
\begin{equation}\label{maineq 1}\sum_{\substack{F \in \mathcal{H}_{q^{2}} \\(F, l)=1 \\ P \mid F \Rightarrow P \notin \mathbb{F}_{q}[T]}} z^{\op{deg}(F)}=\frac{\mathcal{Z}_{q^{2}}(z)}{\mathcal{Z}_{q^{2}}\left(z^{2}\right) \prod_{\substack{P \in \mathbb{F}_{q^{2}}[T] \\ P \mid  l}}\left(1+z^{\operatorname{deg}(P)}\right)}\frac{\prod_{\substack{R\in \mathbb{F}_{q}[T]  \\ \deg(R)\ \text{odd}}}A_R(z)\prod_{\substack{R\in \mathbb{F}_{q}[T]\\ \op{deg}(R)\ \text{even}}}B_R(z)}{\prod_{\substack{R\in \mathbb{F}_{q}[T] \\R \mid l\\ \op{deg}(R)\ \text{odd}}}A_R(z)\prod_{\substack{R\in \mathbb{F}_{q}[T] \\ R \mid l \\ \op{deg}(R)\ \text{even}}}B_R(z)}.\end{equation}

\par Invoking \eqref{maineq 1}, we deduce that:
\begin{equation}\label{main eq 3}
    \begin{split} \sum_{l \in \mathcal{M}_q} d(l^3) u^{\deg(l)} \sum_{\substack{F \in \mathcal{H}_{q^{2}} \\(F, l)=1 \\ P \mid F \Rightarrow P \notin \mathbb{F}_{q}[T]}} z^{\deg(F)}&= \frac{\mathcal{Z}_{q^{2}}(z)\prod_{\substack{R\in \mathbb{F}_{q}[T]  \\ \op{deg}(R)\ \text{odd}}}A_R(z)\prod_{\substack{R\in \mathbb{F}_{q}[T]\\ \op{deg}(R)\ \text{even}}}B_R(z)}{\mathcal{Z}_{q^{2}}\left(z^{2}\right)}  \\\times&
    \sum_{l \in \mathcal{M}_q}  \frac{d(l^3) u^{\deg(l)}}{\prod_{\substack{P \in \mathbb{F}_{q^{2}}[T] \\ P \mid  l}}\left(1+z^{\operatorname{deg}(P)}\right)\prod_{\substack{R\in \mathbb{F}_{q}[T] \\R \mid l\\ \op{deg}(R)\ \text{odd}}}A_R(z)\prod_{\substack{R\in \mathbb{F}_{q}[T] \\ R \mid l \\ \op{deg}(R)\ \text{even}}}B_R(z)},
    \end{split}
\end{equation}
\noindent and we have the following relation
\begin{equation}\label{main eq 2}
\begin{split}
& \sum_{l \in \mathcal{M}_q}  \frac{ d(l^3) u^{\deg(l)} }{ \displaystyle \prod_{\substack{P \in \mathbb{F}_{q^2}[T] \\ P\mid l}} (1+z^{\deg(P)}) \displaystyle \prod_{\substack{R \in \mathbb{F}_q[T] \\ R\mid l \\ \deg(R)  \,\mathrm{ odd}}} A_R(z)  \prod_{\substack{R \in \mathbb{F}_q[T] \\ R\mid l \\ \deg(R)  \,\mathrm{ even}}} B_R(z) } \\
=&\prod_{\substack{R \in \mathbb{F}_q[T] \\ \deg(R)  \,\mathrm{ odd}}} \left[ 1+ \frac{\sum_{j=1}^\infty d(R^{3j})u^{j\deg(R)}}{(1+z^{\deg(R)}) A_R(z) }\right] \prod_{\substack{R \in \mathbb{F}_q[T] \\ \deg(R)  \,\mathrm{ even}}} \left[1+ \frac{\sum_{j=1}^\infty d(R^{3j})u^{j\deg(R)}}{(1+z^{\frac{\deg(R)}{2}})^2 B_R(z)} \right]\\
&=\prod_{\substack{R \in \mathbb{F}_q[T] \\ \deg(R)  \,\mathrm{ odd}}} \left[ 1+ {\sum_{j=1}^\infty (3j+1)u^{j\deg(R)}}\right] \prod_{\substack{R \in \mathbb{F}_q[T] \\ \deg(R)  \,\mathrm{ even}}} \left[1+ \frac{\sum_{j=1}^\infty (3j+1)u^{j\deg(R)}}{(1+2z^{\frac{\deg(R)}{2}})} \right]\\
&= \prod_{\substack{R \in \mathbb{F}_q[T] \\ \deg(R)  \,\mathrm{ odd}}} \left[ 1+ \frac{ u^{\deg(R)}(4-u^{\deg(R)})}{ (1-u^{\deg(R)})^2}\right] \prod_{\substack{R \in \mathbb{F}_q[T] \\ \deg(R)  \,\mathrm{ even}}} \left[1+ \frac{u^{\deg(R)} (4-u^{\deg(R)})}{(1+2z^{\frac{\deg(R)}{2}})(1-u^{\deg(R)})^2} \right],
\end{split}
\end{equation}
where $R$ denotes a monic irreducible polynomial in $\mathbb{F}_q[T]$. 

\par We set 
\[\begin{split}
& \mathcal{B}_1(u, z):=\sum_{l \in \mathcal{M}_q} \frac{d(l^3)}{\sqrt{|\ell|}^3} u^{\deg(l)} \sum_{\substack{F \in \mathcal{H}_{q^{2}} \\(F, l)=1 \\ P \mid F \Rightarrow P \notin \mathbb{F}_{q}[T]}} z^{\deg(F)}\\
& \mathcal{B}_2(u, z):=\sum_{l \in \mathcal{M}_q} d(l^3) u^{\deg(l)} \sum_{\substack{F \in \mathcal{H}_{q^{2}} \\(F, l)=1 \\ P \mid F \Rightarrow P \notin \mathbb{F}_{q}[T]}} z^{\deg(F)},\\ \end{split}\]
and note that 
\[\mathcal{B}_1(q^{3/2}u, z)=\mathcal{B}_2(u, z).\]
Combining \eqref{main eq 3} and \eqref{main eq 2}, one has that
\begin{equation}\label{main eq 4}
\begin{split}
 \mathcal{B}_2(u,z) &=\frac{\mathcal{Z}_{q^{2}}(z)}{\mathcal{Z}_{q^{2}}(z^2)} \prod_{\substack{R \in \mathbb{F}_q[T] \\ \deg(R)  \,\mathrm{ odd}}} \left[  \frac{ (1+2u^{\deg(R)})}{(1+z^{\deg(R)}) (1-u^{\deg(R)})^2}\right] \\& \times \prod_{\substack{R \in \mathbb{F}_q[T] \\ \deg(R)  \,\mathrm{ even}}} \frac{1}{(1+z^{\frac{\deg(R)}{2}})^2}\left[1+2z^{\frac{\deg(R)}{2}}+\frac{ u^{\deg(R)}(4-u^{\deg(R)})}{(1-u^{\deg(R)})^2} \right]\\
  & = \mathcal{Z}_{q}^4{(u)}\frac{\mathcal{Z}_{q^{2}}(z)}{\mathcal{Z}_{q^{2}}(z^2)} A_q(z,u),
\end{split}
\end{equation}
where \begin{equation}\label{A_Qdefn}
  \begin{split}
A_q(z,u)&=\prod_{\substack{R \in \mathbb{F}_q[T] \\ \deg(R)  \,\mathrm{ odd}}} \left[  \frac{ (1-3u^{2\deg(R)}+2u^{3\deg(R)})}{(1+z^{\deg(R)}) }\right] \\ &\times \prod_{\substack{R \in \mathbb{F}_q[T] \\ \deg(R)  \,\mathrm{ even}}} \frac{1}{(1+z^{\frac{\deg(R)}{2}})^2}\left[2z^{\frac{\deg(R)}{2}}(1-u^{\deg(R)})^4+ 1-3u^{2\deg(R)}+2u^{3\deg(R)}\right].
\end{split}  
\end{equation}
\noindent Apply Perron's formula (Theorem \ref{Perron's thm}) twice to the generating series given above and get \eqref{S-princ} as follows:
\[S_{t, \rm{prin}, \tinycube}= \frac{1}{(2\pi i)^2}\int_{|u|= 1/q^{\varepsilon}}\int_{|z|= 1/q^{2+\varepsilon}} \frac{\mathcal{B}_1(u,z)}{z^{g/2+1}(1-u)u^{t_0}} \frac{du}{u}\frac{dz}{z}.\] Substituting $q^{3/2} u$ in place of $u$, we get:

\begin{equation}\label{S_t prin integral formula}\begin{split}S_{t, \rm{prin}, \tinycube}= & \frac{1}{(2\pi i)^2}\int_{|u|= 1/q^{3/2+\varepsilon}}\int_{|z|= 1/q^{2+\varepsilon}} \frac{\mathcal{B}_2(u,z)}{z^{g/2+1}(1-q^{3/2} u)(q^{3/2}u)^{t_0}} \frac{du}{u}\frac{dz}{z} \\ 
=&\frac{1}{(2\pi i)^2}\int_{|u|= 1/q^{3/2+\varepsilon}}\int_{|z|=1/q^{2+\varepsilon}}\frac{A_q(z,u)(1-q^2z^2)}{(1-qu)^4(1-q^2z)(1-q^{3/2}u)z^{g/2+1}(q^{3/2}u)^{t_0}}\frac{d z}{z}\frac{d u}{u}.
\end{split}\end{equation}

\begin{lemma}
   Let $A_q(z,u)$ be the function defined in \eqref{A_Qdefn}. Then, for any fixed value of $u$ with $|u|< \frac{1}{q^{1/2}}$, $A_{q}(z,u)$ is analytic for $|z|<\frac{1}{q}$.
\end{lemma}
\begin{proof}
Write $A_q(z,u)=A_{q,1}(z,u) A_{q,2}(z,u)$, where:
\[A_{q,1}(z,u):=\prod_{\substack{R \in \mathbb{F}_q[T] \\ \deg(R)  \,\mathrm{ odd}}} \left[  \frac{ (1-3u^{2\deg(R)}+2u^{3\deg(R)})}{(1+z^{\deg(R)}) }\right]\]
and
\[A_{q,2}(z,u):=\prod_{\substack{R \in \mathbb{F}_q[T] \\ \deg(R)  \,\mathrm{ even}}} \frac{1}{(1+z^{\frac{\deg(R)}{2}})^2}\left[2z^{\frac{\deg(R)}{2}}(1-u^{\deg(R)})^4+ 1-3u^{2\deg(R)}+2u^{3\deg(R)}\right].
\]
We study the convergence of both terms separately, beginning with $A_{q,1}(z,u)$. We write 
 \[
     \begin{split}
      A_{q,1}(z,u)&=\prod_{\substack{R \in \mathbb{F}_q[T] \\ \deg(R)  \,\mathrm{ odd}}} \left[  \frac{ (1-3u^{2\deg(R)}+2u^{3\deg(R)}))}{(1+z^{\deg(R)}) }\right]\\&=\mathcal{Z}_{q}{(z^2)}\prod_{\substack{R \in \mathbb{F}_q[T] \\ \deg(R)  \,\mathrm{ odd}}}\left((1-3u^{2\deg(R)}+2u^{3\deg(R)})(1-z^{\deg(R)})\right)\times \prod_{\substack{R \in \mathbb{F}_q[T] \\ \deg(R)  \,\mathrm{ even}}}(1-z^{2\deg(R)})
  \end{split}
 \]
We see that the product above converges to a holomorphic function since the series \[\sum_{R \in \mathbb{F}_q[T]}(-3u^{2\deg(R)}+2u^{3\deg(R)}-z^{\deg(R)}+3u^{2\deg(R)}z^{\deg(R)}-2u^{3\deg(R)}z^{\deg(R)})\] converges uniformly for $|z|<{q^{-1}}$ and $|u|<{q^{-1/2}}$. Thus, $A_{q,1}(z,u)$ is analytic for the prescribed region.

\par Next, we consider the product $A_{q,2}(z,u)$. We have the following simplifications:
{\small\[\begin{split}
& A_{q,2}(z,u)\\&=   \prod_{\substack{R \in \mathbb{F}_q[T] \\ \deg(R)  \,\mathrm{ even}}} \frac{1}{(1+z^{\frac{\deg(R)}{2}})^2}\left[2z^{\frac{\deg(R)}{2}}(1-u^{\deg(R)})^4+ 1-3u^{2\deg(R)}+2u^{3\deg(R)}\right]\\
 &=\mathcal{Z}_{q}^{2}{(z)}\prod_{\substack{R \in \mathbb{F}_q[T] \\ \deg(R)  \,\mathrm{ odd}}}{(1-z^{{\deg(R)}})^2} \prod_{\substack{R \in \mathbb{F}_q[T] \\ \deg(R)  \,\mathrm{ even}}} (1+2z^{\deg(R)/2}(1-u^{\deg(R)})^4-3u^{2\deg(R)}+2u^{3\deg(R)})(1-z^{\deg(R)/2})^2.
\end{split}\]}
The product over odd irreducible polynomials is analytic for  $|z|<\frac{1}{q}$. To check the region of convergence for  the product over even irreducible polynomials, we consider the corresponding sum series
{\small\[\begin{split}
   &\sum_{{\substack{R \in \mathbb{F}_q[T] \\ \deg(R)  \,\mathrm{ even}}}}[(1+2z^{\deg(R)/2}(1-u^{\deg(R)})^4-3u^{2\deg(R)}+2u^{3\deg(R)})(1-z^{\deg(R)/2})^2-1]\\ 
   & =\sum_{{\substack{R \in \mathbb{F}_q[T] \\ \deg(R)  \,\mathrm{ even}}}}\left\{(1+2z^{\deg(R)/2})(1-2z^{\deg(R)/2})-1\right\} +O\left(\sum_{{\substack{R \in \mathbb{F}_q[T] \\ \deg(R)  \,\mathrm{ even}}}}\left(z^{\op{deg}(R)}+z^{\op{deg}(R)/2}u ^{\op{deg}(R)}+u^{2\op{deg}(R)}\right)\right)\\
   &=4\sum_{{\substack{R \in \mathbb{F}_q[T] \\ \deg(R)  \,\mathrm{ even}}}}z^{\deg(R)}+O\left(\sum_{{\substack{R \in \mathbb{F}_q[T] \\ \deg(R)  \,\mathrm{ even}}}}\left(z^{\op{deg}(R)}+z^{\op{deg}(R)/2}u ^{\op{deg}(R)}+u^{2\op{deg}(R)}\right)\right).
\end{split}\]}

Therefore, $A_{q,2}(z,u)$ is analytic for $|z|<\frac{1}{q}$ and $|u|<\frac{1}{q^{1/2}}$. 
This proves the lemma.
\end{proof}
Shifting the integral \eqref{S_t prin integral formula} from $|z|=1/q^{2+\varepsilon}$ to $|z|=1/q^{1+\varepsilon}$, we find that there is a pole at $z=1/q^2$. Thus, by Cauchy's residue theorem, we deduce that: 
\begin{align*}
   S_{t, \rm{prin}, \tinycube}=&\frac{q^{g+2}}{\zeta_{q}(3)}\frac{1}{2\pi i}\int_{|u|= 1/q^{3/2+\varepsilon}} \frac{A_q(1/q^2,u)}{(1-qu)^4(1-q^{3/2}u)(q^{3/2}u)^{t_0}}\frac{d u}{u}\\&+O\left(\int_{|u|= 1/q^{3/2+\varepsilon}}\int_{|z|= 1/q^{1+\varepsilon}}\frac{A_q(z,u)(1-q^2z^2)}{(1-qu)^4(1-q^2z)(1-q^{3/2}u)z^{g/2+1}(q^{3/2}u)^{t_0}}\frac{d z}{z}\frac{d u}{u}\right)\\
   &=\frac{q^{g+2}}{\zeta_{q}(3)}\frac{1}{2\pi i}\int_{|u|= 1/q^{3/2+\varepsilon}} \frac{A_q(1/q^2,u)}{(1-qu)^4(1-q^{3/2}u)(q^{3/2}u)^{t_0}}\frac{d u}{u}+O(q^{g/2+\varepsilon g}),
\end{align*}
where the error term is estimated as follows:
\begin{equation}\label{before eqn}
\begin{split}
   & \left|\int_{|u|= 1/q^{3/2+\varepsilon}}\int_{|z|= 1/q^{1+\varepsilon}}\frac{A_q(z,u)(1-q^2z^2)}{(1-qu)^4(1-q^2z)(1-q^{3/2}u)z^{g/2+1}(q^{3/2}u)^{t_0}}\frac{d z}{z}\frac{d u}{u}\right|\\&\leq \int_{|u|= 1/q^{3/2+\varepsilon}}\int_{|z|= 1/q^{1+\varepsilon}}\frac{|A_q(z,u)||(1-q^2z^2)|}{|(1-qu)^4(1-q^2z)(1-q^{3/2}u)z^{g/2+1}(q^{3/2}u)^{t_0}|}\frac{d z}{|z|}\frac{d u}{|u|}\\&
   \ll q^{g/2+1+\varepsilon g }.
\end{split}
\end{equation}
\noindent Shift the integral from $|u| = q^{-3/2 - \varepsilon}$ to $|u| = q^{1/\varepsilon}$. There are two poles in the open annulus $q^{-3/2 - \varepsilon} < |u| < q^{1/\varepsilon}$: one at $u = q^{-3/2}$ of order 1, and the other at $u = q^{-1}$ of order 4. Thus, by the Cauchy residue theorem, we have

\begin{equation}\label{prin-cube}
    \begin{split}
   S_{t, \rm{prin}, \tinycube}&=\frac{q^{g+2}}{\zeta_{q}(3)}\zeta_{q}(3/2)^2A_q(1/q^2,1/q^{3/2})+\frac{q^{g-t_0/2-2}}{\zeta_{q}(3)}P_2(q)+\frac{q^{g-t_0/2-2}}{\zeta_{q}(3)}Q_3(q^{3/2})+O(q^{g-3t_0/2+\varepsilon t_0}), 
\end{split}
\end{equation}
where $P_2(x)$ and $Q_3(x)$ are degree $2$ and degree $3$ polynomials respectively which can be explicitly determined, and the error term is estimated as in \eqref{before eqn}.
\subsection{Contribution from non-cubes}\label{noncube section}
In this section, we estimate the contribution of $S_{t, \mathrm{prin}, \neq \tinycube}$ and show that it is bounded by the error term in \eqref{eqn for main theorem}. Recall from \eqref{princpal eqn} that
\[ S_{t, \rm{prin}, \neq \tinycube}=\sum_{\substack{f\in \cM_{\leq t}\\f\neq \tinycube}}\frac{d(f)}{\sqrt{|f|}}\left(\sum_{\substack{F \in \mathcal{H}_{q^{2}, g / 2+1} \\ P \mid F \Rightarrow P \notin \mathbb{F}_{q}[T]}} \chi_F(f)\right).\]
\noindent From \eqref{mobius identity}, we arrive at the following identity:
\begin{equation*}
S_{t, \rm{prin}, \neq \tinycube}=  \sum_{\substack{f \in \mathcal{M}_{q, \leq t} \\ f \neq \tinycube}} \frac{d(f)}{q^{\deg(f)/2}}  \sum_{\substack{F \in \mathcal{H}_{q^2,\frac{g}{2}+1-\deg(D)} \\ (F,f)=1}} \chi_F(f)\sum_{\substack{D \in \mathbb{F}_q[T] \\ D\mid F}} \mu(D).
\end{equation*}
Replacing $F$ with $F D$ and using the fact that $\chi_D(f) = 1$ (cf. Lemma \ref{chi_D(f)=1 lemma}), we have
\begin{equation}
S_{t, \rm{prin}, \neq \tinycube}=  \sum_{\substack{f \in \mathcal{M}_{q, \leq t} \\ f \neq \tinycube}} \frac{d(f)}{q^{\deg(f)/2}} \sum_{\substack{D \in \mathcal{M}_{q, \leq \frac{g}{2}+1} \\ (D,f)=1}} \mu(D) \sum_{\substack{F \in \mathcal{H}_{q^2,\frac{g}{2}+1-\deg(D)} \\ (F,f)=1}} \chi_F(f).
\label{non_cube_term}
\end{equation}
\noindent Consider the generating series for the sum over $F$: 
\begin{align*}
 \sum_{\substack{F \in \mathcal{H}_{q^2} \\ (F,f)=1}} \chi_F(f) u^{\deg(F)} = \prod_{\substack{P \in \mathbb{F}_{q^2}[T] \\ P \nmid f}} \left ( 1+\chi_P (f) u^{\deg(P)} \right ) = \frac{ \mathcal{L}_{q^2} \left(u, \chi_f\right) }{\mathcal{L}_{q^2}(u^2, \overline{\chi_f})} \prod_{\substack{P \in \mathbb{F}_{q^2}[T]\\ P \mid f}} \left(\frac{  1- \chi_P(f) u^{\deg(P)} }{   1- \overline{\chi_P}(f) u^{2 \deg(P)} }\right).
\end{align*}
\noindent Applying Perron's formula to the generating series above, we have
\begin{equation}\label{boring equation 1}
\begin{split}
 \sum_{\substack{F \in \mathcal{H}_{q^2,\frac{g}{2}+1-\deg(D)} \\ (F,D)=1}} \chi_F(f) = \frac{1}{2 \pi i} \int_{|u|=1/q^{1+\varepsilon}} \frac{ \mathcal{L}_{q^2} \left(u,\chi_f\right) }{\mathcal{L}_{q^2}(u^2, \overline{\chi_f}) u^{\frac{g}{2}+1-\deg(D)}} \prod_{\substack{P \in \mathbb{F}_{q^2}[T] \\ P \mid f}} \left(\frac{  1- \chi_P(f) u^{\deg(P)} }{   1- \overline{\chi_P}(f) u^{2 \deg(P)} } \right)\, \frac{du}{u}.
\end{split}
\end{equation}
By the well known Lindel\"of bound for L-functions (cf. \cite[Lemma 2.5]{DFLCubic}), one has that 
\begin{equation}\label{lindelof bound}
    \left|  \mathcal{L}_{q^2} \left(u,\chi_f\right) \right| \ll q^{2 \varepsilon \deg(f)}.
\end{equation}
We shall also use the following lower bound
\begin{equation}\label{L function lower bound}  \left| \mathcal{L}_{q^2} (u^2, \overline{\chi_f})  \right| \gg q^{-2 \varepsilon\deg(f)},\end{equation}
cf. \cite[Lemma 2.6]{DFLCubic}. Applying \eqref{lindelof bound} (respectively \eqref{L function lower bound}) to the numerator (resp. denominator) of \eqref{boring equation 1}, it follows that
$$  \sum_{\substack{F \in \mathcal{H}_{q^2,\frac{g}{2}+1-\deg(D)} \\ (F,D)=1}} \chi_F(f) \ll q^{\frac{g}{2}-\deg(D)} q^{4 \varepsilon \deg(f)+2 \varepsilon \deg(D)}.$$
\noindent A trivial estimation of the sums over $D$ in \eqref{non_cube_term} yields the upper bound
\begin{equation*}
\sum_{\substack{D \in \mathcal{M}_{q, \leq \frac{g}{2}+1} \\ (D,f)=1}} \mu(D) \sum_{\substack{F \in \mathcal{H}_{q^2,\frac{g}{2}+1-\deg(D)} \\ (F,Df)=1}} \chi_F(f)\ll q^{\frac{g}{2}+\varepsilon g+4 \varepsilon t}.
\end{equation*}
\noindent The generating series for the sum over $f$ is given as
\[\sum_{{f \in \mathcal{M}_{q} }}{d(f)}u^{\deg(f)}=\prod_{P\in \mathbb{F}_{q}[T]}\left(1+\sum_{j=1}^{\infty}d(P^{j})u^{\deg(P^j)}\right)=\prod_{P\in \mathbb{F}_{q}[T]}\left(1+\sum_{j=1}^{\infty}(j+1)u^{\deg(P^j)}\right)= \mathcal{Z}_{q}^2{(u)}.\]
\noindent We write 
\[G(u)=\sum_{{f \in \mathcal{M}_{q} }}\frac{d(f)}{q^{\op{deg}(f)/2}}u^{\deg(f)}\]
and set $u:=q^{1/2} w$. Then, we have that 
\[\tilde{G}(w)=\sum_{{f \in \mathcal{M}_{q} }}d(f)w^{\deg(f)}=\mathcal{Z}_q^2(w).\]
\[\tilde{G}(u/q^{1/2})=G(u)\]
\noindent Applying Perron's formula to $G(u)$, we get 
\[ \sum_{\substack{f \in \mathcal{M}_{q, \leq t} \\ f \neq \tinycube}} \frac{d(f)}{q^{\deg(f)/2}}=  \frac{1}{2 \pi i} \int_{|u|=1/q^{\frac{1}{2}+\varepsilon}}\frac{G(u)}{(1-u)u^{t}}\frac{du}{u} .\]
Replace $u$ with $q^{1/2} u$ to get that the above equals 
\[\sum_{\substack{f \in \mathcal{M}_{q, \leq t} \\ f \neq \tinycube}} \frac{d(f)}{q^{\deg(f)/2}}=  \frac{1}{2 \pi i} \int_{|u|=1/q^{1+\varepsilon}}\frac{\widetilde{G}(q^{1/2}u)}{(1-q^{1/2}u)(q^{1/2}u)^{t}}\frac{du}{u}.\]
Note that $\widetilde{G}(q^{1/2}u)=\mathcal{Z}_q^2(u)=\frac{1}{(1-qu)^2}$. Thus, we get that:
\[ \sum_{\substack{f \in \mathcal{M}_{q, \leq t} \\ f \neq \tinycube}} \frac{d(f)}{q^{\deg(f)/2}}=  \frac{1}{2 \pi i} \int_{|u|=1/q^{1+\varepsilon}}\frac{1}{(1-qu)^2(1-q^{1/2}u)(q^{1/2}u)^{t}}\frac{du}{u} .\]
We thus have the following estimate: 
\begin{equation}\label{prin-non cube}
    \sum_{\substack{f \in \mathcal{M}_{q, \leq t} \\ f \neq \tinycube}} \frac{d(f)}{q^{\deg(f)/2}} \ll q^{t/2+t\varepsilon}.
\end{equation}
Combining all the upper bounds and substituting in \eqref{non_cube_term}, we get the final bound
\begin{equation}\label{final prin-non cube}
    S_{t, \rm{prin}, \neq \tinycube} \ll q^{\frac{g+t}{2}+\varepsilon g+t\varepsilon}.
\end{equation}
Combining the perfect cube \eqref{prin-cube} and non perfect cube \eqref{final prin-non cube} estimates, we arrive at the following bound for the principal sum:
\begin{equation}\label{princ final}
     \begin{split}
   S_{t, \rm{prin} }&=\frac{q^{g+2}}{\zeta_{q}(3)}\zeta_{q}(3/2)^2A_q(1/q^2,1/q^{3/2})+\frac{q^{g-t/6-2}}{\zeta_{q}(3)}P_2(q)+\frac{q^{g-t/6-2}}{\zeta_{q}(3)}Q_3(q^{3/2})+O(q^{g-t/2+\varepsilon t}+q^{\frac{g+t}{2}+\varepsilon g+t\varepsilon}), 
\end{split}
\end{equation}
\subsection{The dual term}\label{dual term section}
Recall the expression of the dual term
\begin{equation}\label{section 3.4 dual term}S_{t, \rm{dual}}=\sum_{f\in \cM_{\leq 2g-t-1}}  \frac{d(f)}{\sqrt{|f|}}\left(\sum_{\substack{F \in \mathcal{H}_{q^{2}, g / 2+1} \\ P \mid F \Rightarrow P \notin \mathbb{F}_{q}[T]}} \omega(\chi_F)^2\bar{\chi}_F(f)\right).\end{equation}
Note that the restriction of $\chi_F$ to $\mathbb{F}_{q}[T]^\times$ is an even character of modulus $F\tilde{F}$, where $\tilde{F}$ is the Galois conjugate of $F$. From \eqref{omega function}, we find that $\omega(\chi_{F})^{2}=q^{-g-2}G^{2}(\chi_{F})$, where
\[
G(\chi_{F}) = \sum_{a \in \mathbb{F}_{q}[T]/(F\tilde{F})} \chi_{F}(a) e_q \left(\frac{a}{F\tilde{F}} \right).
\] 
By the Chinese remainder theorem, there is a natural bijection
$
\beta \mapsto \beta \tilde{F} + \tilde{\beta} F
$ between the residue classes modulo $F$ in $\F_{q^2}[T]$ and the residue classes modulo $F\tilde{F}$ in $\F_q[T]$. Therefore, we can compute the Gauss sum $G(\chi_F)$ as follows:
$$
\begin{aligned}
G(\chi_F) &= \sum_{\beta \in \F_{q^2}[T]/(F)} \chi_F(\beta \tilde{F}) \cdot e_q\left( \frac{\beta \tilde{F} + \tilde{\beta} F}{F\tilde{F}} \right) \\
&= \sum_{\beta \in \F_{q^2}[T]/(F)} \chi_F(\beta) \cdot e_{q^2}\left( \frac{\beta}{F} \right),
\end{aligned}
$$
\noindent where in the second equality we have used that $\chi_F(\tilde{F}) = 1$, a consequence of cubic reciprocity. Thus, one has that $G(\chi_F) = G_{q^2}(1, F)$,
and consequently, \begin{equation}\label{omega gauss sum reln}\omega(\chi_F)^2=q^{-g-2} G_{q^2}(1, F)^2.\end{equation}
Plugging \eqref{omega gauss sum reln} into the formula \eqref{section 3.4 dual term} for the dual term, we deduce that:
\begin{equation}\label{S t dual eqn}S_{t, \rm{dual}}=q^{-g-2}\sum_{f\in \cM_{\leq 2g-t-1}}  \frac{d(f)}{\sqrt{|f|}}\left(\sum_{\substack{F \in \mathcal{H}_{q^{2}, g / 2+1} \\ P \mid F \Rightarrow P \notin \mathbb{F}_{q}[T]}}  G_{q^2}(1,F)^2\bar{\chi}_F(f)\right)\end{equation}
For \((f, F) = 1\), it follows from \eqref{gauss sum relation} that  
\begin{equation}\label{G q blah eqn 1}
    G_{q^2}(1, F)\, \bar{\chi}_F(f) = G_{q^2}(f, F).
\end{equation}
On the other hand, if $(f,F)\neq 1$, then \(\bar{\chi}_F(f) = 0.\)
Moreover, by \cite[Lemma 4.4]{DFLCubic}, if \(F \in \mathbb{F}_q[T]\) is square-free, then  
\begin{equation}\label{G q blah equation 2}
    G_{q^2}(1, F) = q^{\deg F}=q^{g/2+1}.
\end{equation}
Combining \eqref{G q blah eqn 1} and \eqref{G q blah equation 2}, we obtain  
\begin{equation}\label{G q blah equation 3}
    G_{q^2}(1, F)^2 \, \bar{\chi}_F(f) = G_{q^2}(f, F) \, q^{g/2+1}.
\end{equation}
Substitute \eqref{G q blah equation 3} into the expression \eqref{S t dual eqn} to get that 
\[ S_{t, \rm{dual}}=q^{-g/2-1}\sum_{f\in \cM_{\leq 2g-t-1}}  \frac{d(f)}{\sqrt{|f|}}\left(\sum_{\substack{F \in \mathcal{H}_{q^{2}, g / 2+1},\\  P \mid F \Rightarrow P \notin \mathbb{F}_{q}[T], \\ (F,f)=1}} G_{q^2}(f, F)\right).\]
\noindent Then \eqref{mobius identity} implies that
\begin{align*}
    S_{t, \rm{dual}} &=q^{-g/2-1}\sum_{f\in \cM_{\leq 2g-t-1}}  \frac{d(f)}{\sqrt{|f|}}\sum_{\substack{F \in \mathcal{M}_{q^{2}, g / 2+1} \\(F,f)=1}}G_{q^2}(f,F) \sum_{D\mid F}\mu(D)\\
    &=q^{-g/2-1}\sum_{f\in \cM_{\leq 2g-t-1}}  \frac{d(f)}{\sqrt{|f|}}\sum_{\substack{D \in \mathcal{M}_{q^{2}, g / 2+1} \\(D,f)=1}}\mu(D) \sum_{\substack{F \in \mathcal{M}_{q^{2}, g / 2+1-\deg(D)} \\(F,fD)=1}}G_{q^2}(f,DF) \\
    &=q^{-g/2-1}\sum_{f\in \cM_{\leq 2g-t-1}}  \frac{d(f)}{\sqrt{|f|}}\sum_{\substack{D \in \mathcal{M}_{q^{2}, g / 2+1} \\(D,f)=1}}\mu(D) G_{q^2}(f,D)\sum_{\substack{F \in \mathcal{M}_{q^{2}, g / 2+1-\deg(D)} \\(F,fD)=1}}G_{q^2}(Df,F) .
\end{align*}
\noindent From Proposition \ref{Gauss average Propn} we deduce that 
{\small\[\begin{split}
\sum_{\substack{F \in \mathcal{M}_{q^2,g/2+1-\deg(D)}\\(F,fD)=1}}   G_{q^2}(fD,F) & = \delta_{f_2=1} \frac{q^{\frac{4g}{3}+\frac{8}{3}-4\deg(D)-\frac{4}{3}\deg(f_1) - \frac{8}{3} [g/2+1+ \deg(f_1)]_3}}{\zeta_{q^2}(2)}
\overline{G_{q^2}(1,f_1D)} \nonumber\\
&\times \rho(1, [g/2+1+\deg(f_1)]_3)\prod_{\substack{P \in \mathbb{F}_{q^2}[T] \\P\mid fD}} \left (1+ \frac{1}{|P|_{q^2}} \right )^{-1}\nonumber\\
& +O \left(\delta_{f_2=1}q^{\frac{g}{3}+\varepsilon g-\deg(D)(1+2\varepsilon)-\frac{\deg(f)_1}{3}} \right) + \frac{1}{2 \pi i} \oint_{|u|=q^{-2\sigma}} \frac{ \tilde{\Psi}_{q^2}(fD,u)}{u^{g/2+1-\deg(D)}} \, \frac{du}{u},
\end{split}\]}
\noindent We write $S_{t, \text{dual}}=M_t+E_t$, where 
\begin{align*}
M_t= & \frac{q^{5g/6+5/3}}{\zeta_{q^2}(2)} \sum_{f \in \mathcal{M}_{q,\leq g-t-1}} \frac{\delta_{f_2=1} d(f)q^{- \frac{8}{3} [g/2+1+ \deg(f_1)]_3}}{q^{\deg(f)/2+\deg(f_1)/3}} \sum_{\substack{D \in \F_q[T]\\ \deg(D) \leq g/2 + 1\\(D,f)=1}} \mu(D)     q^{-4\deg(D)}
|G_{q^2}(1,D)|^2\\
&\times \rho(1, [g/2+1+\deg(f_1)]_3)\prod_{\substack {P \in \F_{q^2}[T] \\ P\mid fD}} \left (1+ \frac{1}{|P|_{q^2}} \right )^{-1} \\
=& \frac{q^{5g/6+5/3}}{\zeta_{q^2}(2)} \sum_{f \in \mathcal{M}_{q,\leq g-t-1}} \frac{\delta_{f_2=1}d(f)q^{- \frac{8}{3} [g/2+1+\deg(f_1)]_3}}{q^{\deg(f)/2+\deg(f_1)/3}} 
\rho(1, [g/2+1+ \deg(f_1)]_3)\\
&\times \prod_{\substack {P \in \F_{q^2}[T] \\P\mid f}} \left (1+ \frac{1}{|P|_{q^2}} \right )^{-1} \sum_{\substack{D \in \F_q[T]\\ \deg(D) \leq g/2 + 1\\(D,f)=1}} \mu(D)q^{-2\deg(D)}
\prod_{\substack {P \in \F_{q^2}[T] \\P\mid D}} \left (1+ \frac{1}{|P|_{q^2}} \right )^{-1}
\end{align*}
and
 \begin{equation}\label{dual error}
    \begin{split}
    E_t&=O\left(q^{-g/2-1}\sum_{f\in \cM_{\leq 2g-t-1}}  \frac{d(f)}{\sqrt{|f|}}\sum_{\substack{D \in \mathcal{M}_{q^{2}, g / 2+1} \\(D,f)=1}} G_{q^2}(f,D)\delta_{f_2=1}q^{\frac{g}{3}+\varepsilon g-\deg(D)(1+2\varepsilon)-\frac{\deg(f)_1}{3}} \right)\\&+\frac{q^{-g/2-1}}{2 \pi i}\sum_{f\in \cM_{\leq 2g-t-1}}  \frac{d(f)}{\sqrt{|f|}}\sum_{\substack{D \in \mathcal{M}_{q^{2}, g / 2+1} \\(D,f)=1}}\mu(D) G_{q^2}(f,D)  \oint_{|u|=q^{-2\sigma}} \frac{ \tilde{\Psi}_{q^2}(fD,u)}{u^{g/2+1-\deg(D)}} \, \frac{du}{u}.
    \end{split}
\end{equation}
\noindent In order to evaluate the sum over \( D \), we introduce its generating function, given by  
\[
\sum_{\substack{D \in \mathbb{F}_q[T] \\ (D,f)=1}}  \frac{ \mu(D)}{q^{2 \deg(D)}} \prod_{\substack{P \in \mathbb{F}_{q^2}[T] \\ P\mid D}} \left ( 1+ \frac{1}{|P|_{q^2}} \right )^{-1} w^{\deg(D)},
\]
which can be rewritten as the Euler product  
\[
\prod_{\substack{R \in \mathbb{F}_q[T] \\ \deg(R)  \,\mathrm{ odd} \\ R \nmid f}} \left[ 1- \frac{w^{\deg(R)}}{q^{2 \deg(R)} (1+ \frac{1}{q^{2 \deg(R)}})}\right]\prod_{\substack{R \in \mathbb{F}_q[T] \\ \deg(R)  \,\mathrm{ even} \\ R \nmid f}} \left[ 1- \frac{w^{\deg(R)}}{q^{2 \deg(R)} ( 1+ \frac{1}{q^{\deg(R)}})^2}\right].
\]
Here, we have determined the prime divisors in \(\mathbb{F}_{q^2}[T]\) by counting their underlying prime divisors in \(\mathbb{F}_q[T]\). Letting \( A_{\mathrm{dual},R}(w) \) and \( B_{\mathrm{dual},R}(w) \) denote the first and second factors in the above product, respectively, we define  
\[
\mathcal{J}(w) = \prod_{\substack{R \in \mathbb{F}_q[T] \\ \deg(R)  \,\mathrm{ odd}}} 
A_{\mathrm{dual},R}(w) \prod_{\substack{R \in \mathbb{F}_q[T] 
\\ \deg(R)  \,\mathrm{ even}}} 
B_{\mathrm{dual},R}(w).
\]
This series is absolutely convergent for \( |w|<q \). By applying Perron's formula, we obtain  
\[
\sum_{\substack{D \in \F_q[T]\\ \deg(D) \leq g/2 + 1\\(D,f)=1}} \mu(D)     q^{-2\deg(D)}  \prod_{\substack{P \in \mathbb{F}_{q^2}[T] \\ P\mid D}} \left (1+ \frac{1}{|P|_{q^2}} \right )^{-1}  
\]
\[
= \frac{1}{2 \pi i} \oint \frac{\mathcal{J}(w)}{w^{g/2+1}(1-w)} 
\prod_{\substack{R \in \mathbb{F}_q[T] \\ \deg(R)  \,\mathrm{ odd} \\ R \mid f}} A_{\mathrm{dual},R}(w)^{-1} \prod_{\substack{R \in \mathbb{F}_q[T] \\ \deg(R)  \,\mathrm{ even} \\ R\mid f}} B_{\mathrm{dual},R}(w)^{-1} \, \frac{dw}{w}.
\]

Next, we incorporate the summation over \( f \). Substituting the expression obtained above for the sum over \( D \), we find that  
\[
M_t= \frac{q^{5g/6+5/3}}{\zeta_{q^2}(2)} \sum_{f \in \mathcal{M}_{q,\leq g-t-1}} \frac{\delta_{f_2=1}d(f)\rho(1, [g/2+1+ \deg(f_1)]_3)}{q^{\frac{8}{3} [g/2+1+\deg(f_1)]_3}q^{\deg(f)/2+\deg(f_1)/3}} 
\]
\[
\times \prod_{\substack{R \in \mathbb{F}_q[T] \\ \deg(R)  \,\mathrm{ odd} \\ R \mid f}} \left(1+\frac{1}{q^{2\deg(R)}}\right)^{-1}
\prod_{\substack{R \in \mathbb{F}_q[T] \\ \deg(R)  \,\mathrm{ even} \\ R \mid f}} \left(1+\frac{1}{q^{\deg(R)}}\right)^{-2} 
\]
\[
\times \frac{1}{2 \pi i} \oint \frac{\mathcal{J}(w)}{w^{g/2+1}(1-w)}\prod_{\substack{R \in \mathbb{F}_q[T] \\ \deg(R)  \,\mathrm{ odd} \\ R \mid f}} A_{\mathrm{dual},R}(w)^{-1} \prod_{\substack{R \in \mathbb{F}_q[T] \\ \deg(R)  \,\mathrm{ even} \\ R\mid f}} B_{\mathrm{dual},R}(w)^{-1} \, \frac{dw}{w}.
\]
\noindent Let
\begin{align*}
\mathcal{H}(u,w)= &\sum_{f}
\frac{\delta_{f_2=1} d(f)}{q^{\deg(f)/2+\deg(f_1)/3}} 
\prod_{\substack{R \in \mathbb{F}_q[T] \\ \deg(R)  \,\mathrm{ odd} \\ R \mid f}} C_R(w)^{-1} \prod_{\substack{R \in \mathbb{F}_q[T] \\ \deg(R)  \,\mathrm{ even} \\ R\mid f}} D_R(w)^{-1} u^{\deg(f)},
\end{align*}
where 
\[C_R(w)=1+\frac{1}{q^{2\deg(R)}}-\frac{w^{\deg(R)}}{q^{2\deg(R)}}, \qquad D_R(w)=\left(1+\frac{1}{q^{\deg(R)}}\right)^2-\frac{w^{\deg(R)}}{q^{2\deg(R)}}.\]
\noindent Then we can write down an Euler product for $\mathcal{H}(u,w)$ and we have that
{\small \begin{align*}
 \mathcal{H}(u,w)= & \prod_{\substack{R \in \mathbb{F}_q[T] \\ \deg(R)  \,\mathrm{ odd}}} \left[1+C_R(w)^{-1} \left ( \frac{1}{q^{\deg(R)/3}} \sum_{j=0}^{\infty} \frac{ (3j+2)u^{(3j+1) \deg(R)}}{q^{(3j+1) \deg(R)/2}} + \sum_{j=1}^{\infty} \frac{(3j+1) u^{3j \deg(R)}}{q^{3j\deg(R)/2}}\right ) \right] \\
 & \times  \prod_{\substack{R \in \mathbb{F}_q[T] \\ \deg(R)  \,\mathrm{ even}}} \left[1+D_R(w)^{-1} \left ( \frac{1}{q^{\deg(R)/3}} \sum_{j=0}^{\infty} \frac{ (3j+2)u^{(3j+1) \deg(R)}}{q^{(3j+1) \deg(R)/2}} + \sum_{j=1}^{\infty} \frac{ (3j+1)u^{3j \deg(R)}}{q^{3j\deg(R)/2}}\right ) \right].
\end{align*}}
\noindent After simplifying, we have
{\small \begin{align}
\mathcal{H}(u,w) =& \prod_{\substack{R \in \mathbb{F}_q[T] \\ \deg(R)  \,\mathrm{ odd}}} \left[1+C_R(w)^{-1} \left ( \frac{|R|_q^{2/3} u^{\op{deg}(R)}\left(u^{\op{deg}(R)}+2|R|_q^{3/2}\right)}{\left(|R|_q^{3/2}-u^{3(R)}\right)^2} + \frac{ u^{3 \deg(R)}(4|R|_q^{3/2}-u^{3\deg(R)})}{ (|R|_q^{3/2}-u^{3 \deg(R)})^2}\right )\right] \nonumber \\
& \times  \prod_{\substack{R \in \mathbb{F}_q[T] \\ \deg(R)  \,\mathrm{ even}}} \left[1+D_R(w)^{-1} \left ( \frac{|R|_q^{2/3} u^{\op{deg}(R)}\left(u^{\op{deg}(R)}+2|R|_q^{3/2}\right)}{\left(|R|_q^{3/2}-u^{3(R)}\right)^2}+ \frac{ u^{3 \deg(R)}(4|R|_q^{3/2}-u^{3\deg(R)})}{ (|R|_q^{3/2}-u^{3 \deg(R)})^2}\right ) \right] \nonumber \\
=&\mathcal{Z}^2 \left ( \frac{u}{q^{5/6}}\right ) \mathcal{B}(u,w),
\end{align}}
\noindent with $\mathcal{B}(u,w)$ analytic in a wider region (for example, $\mathcal{B}(u,w)$ is absolutely convergent for $|u|<q^{\frac{11}{6}}$ and $|uw|< q^{\frac{11}{6}}$). 

The arguments in \cite{DFLCubic} applied verbatim in this context give us that:
\begin{align*}\label{twoterms}
M_1 =&\frac{q^{5g/6+5/3}}{\zeta_{q^2}(2)} \oint \oint \frac{\mathcal{H}(u,w) \mathcal{J}(w)}{w^{g/2+1}(1-w)} \left[ \frac{1}{u^{g-A-1} (1-u^3)} + \frac{ q^{1/3}}{u^{g-t-3}(1-u^3)} \right] \, \frac{dw}{w} \, \frac{du}{u},
\end{align*}
where we integrate along small circles around the origin. We first shift the contour over $w$ to $|w|=q^{1-\varepsilon}$ (since $\mathcal{J}(w)$ is absolutely convergent for $|w|<q$) and encounter the pole at $w=1$. Note that $\mathcal{H}(u,1)$ has a pole at $u=q^{-1/6}$.

\par Setting
 \begin{equation}
 \mathcal{K}(u)= \mathcal{B}(u,1) \mathcal{J}(1),
 \end{equation}we find that
\begin{align*}
M_t =&\frac{q^{5g/6+5/3}}{\zeta_{q^2}(2)} \oint \frac{ \mathcal{K}(u)}{(1-uq^{1/6})^2(1-u^3) u^{g-t-1}} (1+ q^{1/3}u^2) \, \frac{du}{u}\\
&+ \frac{q^{5g/6+5/3}}{\zeta_{q^2}(2)} \oint_{|u|=q^{-1/6-\varepsilon}} \oint_{|w|=q^{1-\varepsilon}} \frac{\mathcal{H}(u,w) \mathcal{J}(w)}{w^{g/2+1}(1-w)} \left[ \frac{1}{u^{g-t-1} (1-u^3)} + \frac{ q^{1/3}}{u^{g-t-3}(1-u^3)} \right] \, \frac{dw}{w} \, \frac{du}{u}\\
=&\frac{q^{5g/6+5/3}}{\zeta_{q^2}(2)} \oint \frac{ \mathcal{K}(u)}{(1-uq^{1/6})^2(1-u^3) u^{g-t-1}} (1+ q^{1/3}u^2) \, \frac{du}{u}+O\left(q^{\frac{g}{2}-\frac{A}{6}+\varepsilon g}\right).
\end{align*}
\noindent We compute the residue at $q^{-1/6}$:
{\small\[\begin{split}& \lim_{u\rightarrow q^{-1/6}}\left\{\frac{\partial}{\partial u}\left(\frac{(u-q^{-1/6})^2 \mathcal{K}(u)}{(1-uq^{1/6})^2(1-u^3) u^{g-t}} (1+ q^{1/3}u^2)\right)\right\}\\ 
=& \frac{1}{q^{1/3}} \lim_{u\rightarrow q^{-1/6}}\left\{\frac{\partial}{\partial u}\left(\frac{\mathcal{K}(u)(1+q^{1/3}u^2)}{u^{g-t}(1-u^3)}\right)\right\} \\
= & \frac{1}{q^{1/3}} \lim_{u\rightarrow q^{-1/6}} \left\{ \frac{\mathcal{K}'(u)(1+q^{1/3} u^2)}{u^{g-t}(1-u^3)}+\frac{2\mathcal{K}(u) q^{1/3} u}{u^{g-t}(1-u^3)}-\frac{\mathcal{K}(u)u^{g-A-1}(1+q^{1/3}u^2)((g-t)-(g-t+3)u^3)}{(u^{g-t}-u^{g-t+3})^2}\right\}\\
=& q^{-1/3}\left\{ \frac{\mathcal{K}'(q^{-1/6}) q^{g/6-t/6+1/2}}{(\sqrt{q}-1)}+\frac{2\mathcal{K}(q^{-1/6})q^{g/6-t/6+2/3}}{(\sqrt{q}-1)}-\frac{2 q^{g/6-t/6+2/3}\mathcal{K}(q^{-1/6})\left((g-t)\sqrt{q}-(g-t+3)\right)}{(\sqrt{q}-1)^2}\right\}.
\end{split}\]}

\noindent Shifting the contour of integration to $|u|=q^{-\varepsilon}$, we pick up a residue at $u=q^{1/6}$, thus obtaining
{\small\[
\begin{split}
M_{t}= & -q^{-1/3}\left\{ \frac{\mathcal{K}'(q^{-1/6}) q^{g/6-t/6+1/2}}{(\sqrt{q}-1)}-\frac{2 q^{g/6-t/6+2/3}\mathcal{K}(q^{-1/6})\left((g-t)\sqrt{q}-(g-t+3)\right)}{(\sqrt{q}-1)^2}\right.\\& \left.+\frac{2\mathcal{K}(q^{-1/6})q^{g/6-t/6+2/3}}{(\sqrt{q}-1)}\right\}+O\left(q^{\frac{g}{2}-\frac{t}{6}+\varepsilon g}\right) 
+\frac{q^{5 g / 6+5 / 3}}{\zeta_{q^{2}}(2)} \oint_{|u|=q^{-\varepsilon}} \frac{\mathcal{K}(u)\left(1+q^{1 / 3} u^{2}\right)}{\left(1-u q^{1 / 6}\right)^2\left(1-u^{3}\right) u^{g-t-1}} \frac{d u}{u} \\
= & -q^{-1/3}\left\{ \frac{\mathcal{K}'(q^{-1/6}) q^{g/6-t/6+1/2}}{(\sqrt{q}-1)}-\frac{2 q^{g/6-t/6+2/3}\mathcal{K}(q^{-1/6})\left((g-t)\sqrt{q}-(g-t+3)\right)}{(\sqrt{q}-1)^2}\right.\\
&\left.+\frac{2\mathcal{K}(q^{-1/6})q^{g/6-t/6+2/3}}{(\sqrt{q}-1)}\right\} +O\left(q^{\frac{5 g}{6}+\varepsilon g}\right).
\end{split}
\]}
Consider the error term $E_1$ from equation \eqref{dual error}. The first term coming from the first error in equation \eqref{dual error} will be bounded by
\begin{align*} 
\ll &q^{-\frac{g}{2}}\sum_{f \in \mathcal{M}_{q,\leq g-t-1}}\frac{d(f)}{q^{\deg(f)/2}} \sum_{\deg(D)\leq g/2+1} q^{\deg(D)} q^{\frac{g}{3}+\varepsilon g-\deg(D)(1+2\varepsilon)-\frac{\deg(f)_1}{3}}
\ll q^{\left(\frac{1}{2}+2\varepsilon\right)g -\frac{t}{6}}
\end{align*}
\noindent Recall that $2/3<\sigma<4/3$, therefore here we choose $\sigma=4/3-\varepsilon$ and deduce that
{\small\begin{align*}
E_t \ll &q^{-g/2-1} \frac{1}{2 \pi i} \oint_{|u|=q^{-2\sigma}} \sum_{f \in \mathcal{M}_{q,\leq g-t-1}} \frac{d(f)}{q^{\deg(f)/2}} \sum_{\substack{D \in \mathbb{F}_q[T] \\ \deg(D) \leq g/2+1 \\ (D,f)=1}} \mu(D) G_{q^2}(f,D) \frac{\tilde{\Psi}_{q^2}(fD,u)}{u^{g/2+1-\deg(D)}} \, \frac{du}{u} +q^{\frac{g}{2} - \frac{t}{6} + \varepsilon g}\\
\ll& q^{-\frac{g}{2}}\sum_{f \in \mathcal{M}_{q,\leq g-t-1}}\frac{d(f)}{q^{\deg(f)/2}} \sum_{\deg(D)\leq g/2+1} q^{\deg(D)} q^{\sigma g-3 \sigma \deg(D)+\frac{3}{2}\deg(D)+\deg(f) ( \tfrac{3}{2}-\sigma)}+q^{\frac{g}{2} - \frac{t}{6} + \varepsilon g}\\
\ll& g^2 q^{\frac{3}{4}g -\frac{2t}{3}+\varepsilon (g+t)}+q^{\frac{g}{2} - \frac{t}{6} + \varepsilon g}.
\end{align*}} Combining the expressions for $M_t$ and $E_t$ it follows that
{\small\begin{equation}\label{dual final}
\begin{split}
    S_{t,\mathrm{dual}} =& -q^{-1/3}\left\{ \frac{\mathcal{K}'(q^{-1/6}) q^{g/6-t/6+1/2}}{(\sqrt{q}-1)}-\frac{2 q^{g/6-t/6+2/3}\mathcal{K}(q^{-1/6})\left((g-t)\sqrt{q}-(g-t+3)\right)}{(\sqrt{q}-1)^2}\right.\\
&\left.+\frac{2\mathcal{K}(q^{-1/6})q^{g/6-t/6+2/3}}{(\sqrt{q}-1)}\right\} +O\left(q^{\frac{5 g}{6}+\varepsilon g}+g^2 q^{\frac{3}{4}g -\frac{2t}{3}+\varepsilon (g+t)}+q^{(g+A)/2}\right).
\end{split}
\end{equation}}

\subsection{Proof of the main result}
\par We now proceed to prove the main result, combining the estimates for the principal terms from \eqref{princ final}, derived in sections \ref{main tem section} and \ref{noncube section}, with the bounds for the dual terms given in \eqref{dual final} from section \ref{dual term section}.

\begin{proof}[Proof of Theorem \ref{main theorem}]
    \par Recall that from \eqref{princ+dual},
\begin{equation*}
    \sum_{\substack{F \in \mathcal{H}_{q^{2}, g / 2+1} \\ P \mid F \Rightarrow P \notin \mathbb{F}_{q}[T]}} L_{q}\left(1 / 2, \chi_{F}\right)^2 
  = \left(1-\frac{1}{\sqrt{q}}\right)^2\left(\sum_{i=0}^{A}(i+1)S_{A-i, \rm{prin}}+\sum_{i=0}^{2g-A-1}(i+1)S_{2g-A-1+i, \rm{dual}}\right). 
\end{equation*}
\noindent Substituting the estimate $ S_{t,\mathrm{prin}}$  from \eqref{princ final} in the first sum gives
{\small\begin{align*}
    \sum_{i=0}^{A}(i+1)S_{A-i, \rm{prin}}&= \frac{q^{g+2}}{\zeta_{q}(3)}\zeta_{q}(3/2)^2A_q(1/q^2,1/q^{3/2})\sum_{i=0}^{A}(i+1)+ \frac{q^{g-A/6-2}}{\zeta_{q}(3)}(P_2(q)+Q_3(q^{3/2}))\sum_{i=0}^{A}(i+1){q^{i/6}}\\&+O\left(q^{g-A/2+\varepsilon A}\sum_{i=0}^{A}(i+1)q^{i/2-\varepsilon i}+q^{\frac{g+A}{2}+\varepsilon g+A\varepsilon}\sum_{i=0}^{A}(i+1)q^{-i/2}\right)\\
    &= \frac{q^{g+2}}{\zeta_{q}(3)}\zeta_{q}(3/2)^2A_q(1/q^2,1/q^{3/2})\frac{A(A+1)}{2}+\frac{q^{g-A/6-2}}{\zeta_{q}(3)}(P_2(q)+Q_3(q^{3/2}))\\&\times \frac{q^{1/6}(q^{(A+1)/6}-1)(q^{1/6}-2)+((A+2)q^{(A+1)/6}-1)(q^{1/6}-1)}{(q^{1/6}-1)^2}+O\left(q^{g}+q^{(g+A)/2+\varepsilon g}\right).
\end{align*}}
\noindent From \eqref{dual final}, the sum of dual terms is 
{\small \begin{align*}
   & \sum_{i=0}^{2g-A-1}(i+1)S_{2g-A-1+i, \rm{dual}}\\&=-q^{(2g-A-3)/6}\left\{ \frac{\mathcal{K}'(q^{-1/6}) q^{g/2-A/2+1/3}}{(\sqrt{q}-1)}+\frac{2\mathcal{K}(q^{-1/6})q^{g/2-A/2+1/2}}{(\sqrt{q}-1)}\right\}\sum_{i=0}^{2g-A-1}(i+1)q^{-i/6}\\
&+q^{(2g-A-3)/6}\left\{\frac{2 q^{g/2-A/2+1/2}\mathcal{K}(q^{-1/6})}{(\sqrt{q}-1)^2} \right\}\sum_{i=0}^{2g-A-1}(i+1)\left((A-g+1-i)\sqrt{q}-(A-g+4-i)\right)q^{-i/6}\\&+O\left(q^{\frac{5 g}{6}+\varepsilon g}+g^2 q^{\frac{3}{4}g -\frac{2}{3}(2g-1-A)+\varepsilon g}\sum_{i=0}^{2g-A-1}(i+1)q^{\frac{2}{3}i}\right)\\
&=O\left(q^{g/2-A/2+1/3}+q^{5 g/6+\varepsilon g}\right).
\end{align*}}
We set $A:=g/2$, this gives the required result.
\end{proof}

\bibliographystyle{alpha}
\bibliography{references}
\end{document}